\documentclass[12pt,a4paper]{article}
\usepackage{ucs}
\usepackage[applemac]{inputenc}
\usepackage[T1]{fontenc}
\usepackage[english]{babel}		
\usepackage{
	amsmath,
	amsthm,
	amssymb,
	amsfonts,
	mathtools
}
\usepackage[dvips]{hyperref}
\hypersetup{	
	colorlinks,
	linkcolor=black,
	citecolor=black,
	urlcolor=black,
	breaklinks=true
}
\DeclareMathAlphabet{\mathpzc}{OT1}{pzc}{m}{it}
\usepackage{geometry}
\newtheoremstyle{lemma}{\topsep}{\topsep}
	{\itshape}
	{}
	{\bfseries}
	{.}
	{\newline}
	{\thmname{#1}\thmnumber{ #2}\thmnote{ #3}}	
\theoremstyle{lemma}
	\newtheorem{lemma}{Lemma}[section]
	\newtheorem{proposition}[lemma]{Proposition}
	\newtheorem{corollary}[lemma]{Corollary}
	\newtheorem{theorem}[lemma]{Theorem}

\newtheoremstyle{definition}{\topsep}{\topsep}
	{}
	{}
	{\bfseries}
	{.}
	{\newline}
	{\thmname{#1}\thmnumber{ #2}\thmnote{ #3}}	
\theoremstyle{definition}
	\newtheorem{definition}[lemma]{Definition}
	\newtheorem{remark}[lemma]{Remark}
	\newtheorem{example}[lemma]{Example}

\newcommand{\he}{\ensuremath{\alpha}}

\newcommand{\N}{\ensuremath{\mathbb{N}}}

\newcommand{\R}{\ensuremath{\mathbb{R}}}

\newcommand{\Sphere}{\ensuremath{\mathbb{S}}}

\newcommand{\HM}{\ensuremath{\mathcal{H}}}

\newcommand{\F}{\ensuremath{\mathcal{F}}}

\newcommand{\U}{\ensuremath{\mathcal{U}}}
\newcommand{\LM}{\ensuremath{\mathcal{L}}}
\newcommand{\A}{\ensuremath{\mathcal{A}}}
\newcommand{\B}{\ensuremath{\mathcal{B}}}

\newcommand{\V}{\ensuremath{\mathcal{V}}}

\newcommand{\M}{\ensuremath{\mathcal{M}}}
\newcommand{\1}{\mathbb{1}}
\newcommand{\0}{\mathbb{0}}
\newcommand{\measurable}{\ensuremath{\mathcal{C}}}
\newcommand{\mhyphen}{\ensuremath{\,\text{--}\,}}

\DeclarePairedDelimiter\abs{\lvert}{\rvert}
\DeclarePairedDelimiter\norm{\lVert}{\rVert}

\newcommand{\dd}{\ensuremath{\,\text{d}}}

\newcommand{\dx}{\ensuremath{\,\text{d}x}}

\newcommand{\dHM}{\ensuremath{\,\text{d}\HM^{\he}}}

\newcommand{\dL}{\ensuremath{\,\text{d}\mathcal{L}}}

\DeclareMathOperator{\dist}{dist}
\DeclareMathOperator{\diam}{diam}

\newcommand{\rest}{\ensuremath{\vert}}

\renewcommand{\phi}{\varphi}
\renewcommand{\epsilon}{\varepsilon}

\hypersetup{
	pdftitle={Tangency properties of sets with finite geometric curvature energy},
	pdfsubject={},
	pdfauthor={Sebastian Scholtes},
	pdfkeywords={},
	pdfcreator={},
	pdfproducer={}
}

\makeindex
\begin{document}

\title{Tangency properties of sets with finite geometric curvature energies}
\author{\href{mailto:sebastian.scholtes@rwth-aachen.de}{Sebastian Scholtes}}
\date{\today}
\maketitle

\begin{abstract} 
	We investigate inverse thickness $1/\Delta$ and the integral Menger curvature energies
	$\mathcal{U}_{p}^{\he}$, $\mathcal{I}_{p}^{\he}$ and $\mathcal{M}_{p}^{\he}$, to find that finite $1/\Delta$ or $\mathcal{U}_{p}^{\he}$ implies the existence of 
	an \emph{approximate $\he$-tangent} at all points of the set, when $p\geq \he$ and that finite $\mathcal{I}_{p}^{\he}$ or $\mathcal{M}_{p}^{\he}$ implies the existence of a \emph{weak approximate
	$\he$-tangent} at every point of the set for $p\geq 2\he$ or $p\geq 3\he$, respectively, if some additional density properties hold. 
	This includes the scale invariant case $p=2$ for $\mathcal{I}_{p}^{1}$ and $p=3$ for $\mathcal{M}_{p}^{1}$, for which, to the best of our knowledge, no regularity properties are established up to now.
	Furthermore we prove that for $\he=1$ these exponents are sharp, i.e., that if $p$ lies below the
	threshold value of scale innvariance, then there exists a set containing points with no (weak) approximate $1$-tangent, but such that the corresponding energy is still finite.
	For $\mathcal{I}_{p}^{1}$ and $\mathcal{M}_{p}^{1}$ we give an example of a set which possesses a point that has no approximate $1$-tangent, but finite energy for all 
	$p\in (0,\infty)$ and thus show that the existence of weak approximate $1$-tangents is the most we can expect, in other words our results are also optimal in this respect.
\end{abstract}
\centerline{\small Mathematics Subject Classification (2000): 28A75}

\section{Introduction}

In \cite{Leger1999a} J. C. L{\'e}ger was able to show a remarkable theorem\footnote{L{\'e}ger refers to an unplublished work of G. David that had inspired his work and that he took as a guideline for the proof.}, 
which states that one-dimensional Borel sets in $\R^{n}$ with finite integral 
Menger curvature $\mathcal{M}_{2}^{1}$ are $1$-rectifiable. Here, integral Menger\footnote{Named after Karl Menger, because in \cite{Menger1930a}
Menger introduced the limit of the inverse circumradius, when the three points in the argument converge to a single point, 
as a pointwise curvature.}
curvature of a set $X\subset\R^{n}$ refers to the triple integral
over the squared inverse circumradius\footnote{For other applications of the circumradius see \cite{Scholtes2012a}.}, i.e. 
\begin{align*}
	\mathcal{M}_{p}^{\he}(X)&\vcentcolon=\int_{X}\int_{X}\int_{X}[r(x,y,z)]^{-p}\dHM(x)\dHM(y)\dHM(z),
\end{align*}
for $p=2$ and $\he=1$. The circumradius $r(x,y,z)$ is the radius of the unique circle on which the vertices of the non-degenerate triangle $\{x,y,z\}$
lie; in case of a degenerate triangle it is set to be infinite.
These results for $\M_{2}^{1}$ were later extended to metric spaces in \cite{Hahlomaa2008a}, and in \cite{Lin2001a} to sets of fractional dimension, where
$C^{1}$-$\he$-rectifiability of measurable sets with positive and finite $\HM^{\he}$ measure could be shown
if $\M_{2\he}^{\he}$ is finite and $\he\in(0,1/2]$ under the additional assumption that these sets are $\he$-Ahlfors regular\footnote{It was also shown that these results are sharp, i.e. wrong for $s\in(1/2,1)$, but 
that there is no hope of maintaining these results for $s\in (0,1)$ if one drops the $\he$-Ahlfors regularity.}.
As a consequence L{\'e}ger's theorem also ensures
that an $\HM^{1}$ measurable set $E\subset \R^{n}$ with $\mathcal{M}_{2}^{1}(E)<\infty$ has approximate $1$-tangents at $\HM^{1}$ a.e point. 
By an \emph{approximate $\he$-tangent} at a point $x$
we mean a direction $s\in\mathbb{S}^{n-1}$, such that
\begin{align*}
	\lim_{r\downarrow 0}\frac{\HM^{\he}([X\backslash C_{s,\epsilon}(x)]\cap \overline B_{r}(x))}{(2r)^{\he}}=0\quad\text{for all }\epsilon>0,
\end{align*}
where $C_{s,\epsilon}(x)$ is the double cone with opening angle $\epsilon$ in direction $s$ about $x$, cf. \cite[p.203]{Mattila1995a}; for different tangential regularity properties compare also to \cite{Martin1988a}. 
One might think of it as a kind of geometric or measure theoretic counterpart to differentiability. Roughly speaking it means that
the set is locally well approximated by the approximate tangent.
For example a regular, differentiable curve has approximate $1$-tangents at all points and these tangents coincide with the usual tangent, but 
the arc length parametrisation of the set $S\vcentcolon=\{(x,0)\mid x\in [0,1]\}\cup\{(x,x^{2})\mid x\in [0,1] \}$ has no tangent
at $(0,0)$, despite the set having an approximate $1$-tangent at this point, see Remark \ref{differenceapptangentsandtangents}.\\

Now one could ask if the condition $\M_{2}^{1}(X)<\infty$ also guarantees that the set has approximate $1$-tangents at \emph{all} points, 
or, if this is not the case, which influence, if any, the exponent $p$ of the energy $\M_{p}^{1}$ has on these matters. 
This question and related topics are the subject of this paper.\\

Complementary to this research, where highly irregular sets are permitted, was the investigation of rectifiable curves,
which have a classic tangent $\HM^{1}$ a.e. to begin with, of finite $\M_{p}^{1}$ energy. It turns out, see \cite{Strzelecki2010a},
that for $p>3$ this guarantees that the curve is simple and that the arc length parametrisation is of class $C^{1,1-3/p}$,
which can be interpreted as a geometric Morrey-Sobolev imbedding.
In \cite{Blatt2011a} it could be shown that the space of curves with finite $\mathcal{M}_{p}^{1}$ for $p>3$ is that of Sobolev Slobodeckij 
embeddings of class $W^{2-2/p,p}$. 
The same program has also been conducted for a different kind of energy, the so-called tangent point energy in
\cite{Strzelecki2010c,Blatt2011b}.\\

We would like to point out the important role of integral Menger curvature for $p=2$ in the solution of the Painlev\'e problem
i.e. to find geometric characterisations of removable sets for 
bounded analytic functions, see \cite{Pajot2002a,Dudziak2010a} for a detailed presentation and references.\\

Besides integral Menger curvature there are other interesting curvature energies that have been investigated in the same vein.
In \cite{Gonzalez1999a} Gonzales and Maddocks proposed their notion of \emph{thickness}
\begin{align*}
	\Delta[X]\vcentcolon=\inf_{\substack{x,y,z\in X\\x\not=y\not=z\not=x}}r(x,y,z)
\end{align*}
of a knot $X$, which is the infimum of the circumradius $r(x,y,z)$ over all 
triangles $\{x,y,z\}$ on the curve, and also encouraged to investigate different integral curvature energies
\begin{align*}
	\mathcal{U}_{p}^{\he}(X)&\vcentcolon=\int_{X}[\inf_{\substack{y,z\in X\\x\not=y\not=z\not=x}}r(x,y,z)]^{-p}\dHM(x),\\
	\mathcal{I}_{p}^{\he}(X)&\vcentcolon=\int_{X}\int_{X}[\inf_{\substack{z\in X\\x\not=y\not=z\not=x}}r(x,y,z)]^{-p}\dHM(x)\dHM(y),
\end{align*}
and $\mathcal{M}_{p}^{\he}$, where the inverse circumradius is integrated to the power $p$ and the infimisations are successively replaced by integrations. 
That arc length parametrisations of curves with finite inverse thickness are actually of class $C^{1,1}$ and the
existence of ideal knots, which are minimizers of the inverse thickness in a knot class under the restriction of fixed length, 
was shown in \cite{Gonzalez2002b,Cantarella2002a,Gonzalez2003a}; for further research in this direction see also \cite{Schuricht2003a,Schuricht2004a}.
In the series of works \cite{Strzelecki2007a,Strzelecki2009a,Strzelecki2010a} the integral curvature energies $\mathcal{U}_{p}^{1}$,
$\mathcal{I}_{p}^{1}$ and $\mathcal{M}_{p}^{1}$ have been investigated for closed rectifiable curves, to find that the arc length parametrisations
of curves with finite energy for $p\in [1,\infty)$, $p\in (2,\infty)$ and $p\in (3,\infty)$, respectively, are simple and actually belong to the class 
$C^{1,\beta_{\mathcal{\F}}(p)}$, where $\beta_{\mathcal{U}}(p)=1-1/p$, $\beta_{\mathcal{I}}(p)=1-2/p$ and $\beta_{\mathcal{M}}(p)=1-3/p$.
In \cite{Blatt2011a} it could be shown that the space of curves with finite $\mathcal{I}_{p}^{1}$ for $p>2$ and $\mathcal{M}_{p}^{1}$ for $p>3$ is that of Sobolev Slobodeckij embeddings of class
$W^{2-1/p,p}$ and $W^{2-2/p,p}$, respectively.
Similar kind of energies for surfaces and higher dimensional sets have been examined in
\cite{Strzelecki2005a,Strzelecki2006a,Strzelecki2009b,Lerman2009a,Lerman2011a,Kolasinski2011a,Kolasinski2011b,Strzelecki2011a,Blatt2011e}.\\

As mentioned in the very beginning the purpose of this paper is to investigate which pointwise tangential properties can be
expected of sets in Euclidean space with finite energy. To be more precise we will investigate if a set $X$ possesses an approximate $\he$-tangent
or at least a \emph{weak approximate $\he$-tangent} at every point $x$. A weak approximate $\he$-tangent is a mapping $s:(0,\rho)\to \mathbb{S}^{n-1}$, 
such that 
\begin{align*}
	\lim_{r\downarrow 0}\frac{\HM^{\he}([X\backslash C_{s(r),\epsilon}(x)]\cap \overline B_{r}(x))}{(2r)^{\he}}=0\quad\text{for all }\epsilon>0.
\end{align*}
For the example of the T-shaped set $E\vcentcolon=([-1,1]\times \{0\})\cup (\{0\}\times [0,1])$ it is shown that $\mathcal{M}_{2}^{1}(E)<\infty$ does \emph{not} 
suffice to infer that the set has weak approximate $1$-tangents at all points with positive lower density,
see Lemma \ref{summaryforsetE}. So it seems that these properties might depend on the exponent $p$ and the parameter $\he$ of the integral curvature energies $\mathcal{U}_{p}^{\he}$, $\mathcal{I}_{p}^{\he}$ 
and $\mathcal{M}_{p}^{\he}$. Thus our aim is to find conditions on $p$ and $\he$ that ensure the existence of $\he$-tangents at all points with positive lower density.
We shall solve this question thoroughly, to be honest with one minor additional technical requirement in case of $\mathcal{M}_{p}^{\he}$, namely $\Theta^{*\he}(\HM^{\he},X,x)<\infty$,
that, despite our best efforts, we were not able to remove. We have gathered the findings from different sections of the present paper in the following Theorem.
Note that compared to \cite{Lin2001a} we do not require the set to be measurable and $\he$-Ahlfors regular and have more detailed information on \emph{which} points do possess tangents, but we pay for that 
by a more restrictive requirement on the parameter $p$.
We also want to remark that in \cite[1.5 Corollary, p.13]{Lin1997a} it is shown that for $\he>1$ and an $\HM^{\he}$ measurable set $X\subset \R^{n}$ with
$0<\HM^{\he}(X)<\infty$ we always have $\M_{2\he}^{\he}(X)=\infty$, which somewhat restricts the extent of the next theorem for $\he>1$. On the other hand, however, there are a lot more sets allowed in the theorem
that still could have finite $\M_{2\he}^{\he}$.

\begin{theorem}[(Main result)]
	Let $X\subset\R^{n}$, $x\in \R^{n}$, $\he\in (0,\infty)$.
	\begin{itemize}
		\item
			Let $1/\Delta[X]<\infty$, $\HM^{1}(X)<\infty$. Then $X$ has an approximate $1$-tangent at $x$.
		\item
			Let $p\in[\he,\infty)$, $\mathcal{U}_{p}^{\he}(X)<\infty$.
			Then $X$ has an approximate $\he$-tangent at $x$.
		\item
			Let $p\in[2\he,\infty)$, $\mathcal{I}_{p}^{\he}(X)<\infty$ and $\Theta_{*}^{\he}(\HM^{\he},X,x)>0$.
			Then $X$ has a weak approximate $\he$-tangent at $x$.
		\item
			Let $p\in[3\he,\infty)$, $\mathcal{M}_{p}^{\he}(X)<\infty$ and $0<\Theta_{*}^{\he}(\HM^{\he},X,x)\leq\Theta^{*\he}(\HM^{\he},X,x)<\infty$.
			Then $X$ has a weak approximate $\he$-tangent at $x$.
	\end{itemize}
\end{theorem}

To the best of our knowledge these are the first results regarding regularity that incorporate the critical cases $p=2$ for $\mathcal{I}_{p}^{1}$
and $p=3$ for $\mathcal{M}_{p}^{1}$.
Moreover, we show that the exponents are sharp for $\he=1$, that is, there is a set, namely the T-shaped set $E$ from above, that contains a point without weak approximate $1$-tangent and has finite energy if $p$ 
is below the respective threshold value.

\begin{proposition}[(Exponents are sharp for $\he=1$)]
	For $E\vcentcolon=([-1,1]\times \{0\})\cup (\{0\}\times [0,1])$ we have
	\begin{itemize}
		\item
			$\mathcal{U}_{p}^{1}(E)<\infty$ for $p\in (0,1)$,
		\item
			$\mathcal{I}_{p}^{1}(E)<\infty$ for $p\in (0,2)$,
		\item
			$\mathcal{M}_{p}^{1}(E)<\infty$ for $p\in (0,3)$.
	\end{itemize}
\end{proposition}

Furthermore we demonstrate that there is a set $F$ that has a point without an approximate $1$-tangent and finite $\mathcal{I}_{p}^{1}$ and $\mathcal{M}_{p}^{1}$ for all $p\in (0,\infty)$. Hence there is no hope
of obtaining the main result for approximate $1$-tangents instead of weak approximate $1$-tangents for these two energies.

\begin{proposition}[(Weak approximate $1$-tangents are optimal for $\he=1$)]\label{introductionweakapprtangentsoptimal}
	There is a set $F$, $x\in \R^{n}$, such that $F$ has no approximate $1$-tangent at $x$ and
	\begin{itemize}
		\item
			$\mathcal{I}_{p}^{1}(F)<\infty$ for $p\in (0,\infty)$,
		\item
			$\mathcal{M}_{p}^{1}(F)<\infty$ for $p\in (0,\infty)$.
	\end{itemize}
\end{proposition}

To put these results into perspective, we give some simple examples of what they infer, in case of the integral Menger curvature $\M_{p}^{1}$. 
For a curve and $p\geq 3$ we only obtain weak approximate $1$-tangents, which is hardly a new result, except for the case $p=3$, 
as we already knew for $p>3$ that these curves are actually in $W^{2-2/p,p}$ by \cite{Strzelecki2010a,Blatt2011a}, but now we also know that
it is possible for more general connected compact objects to have finite integral Menger curvature for $p\in (0,3)$, objects that cannot be parametrised by a simple curve, like the set $E$. 
On the other hand there are objects with infinite energy, which for instance are constructed by the following principle:
let $X$ be a set, $x\in X$ with positive density and an approximate $1$-tangent $s$ at $x$, further let $\phi_{A}(X)$ be the set $X$
rotated by a rotation matrix $A$ about the point $x$ in such a way that the axis of rotation does not coincide with $s$. Then the set $X\cup\phi_{A}(X)$ -- for example a polygon with two edges -- 
has no weak approximate $1$-tangent at $x$, see Lemma \ref{constructionnoweakapprtangent}, and hence infinite $\M_{p}^{1}$ for $p\geq 3$. Here previously no conclusive statement was possible.
In addition, for $p\in(0,3)$ all polygons have finite $\M^{1}_{p}$ energy, which can be seen using techniques of this paper. Similar statements hold for $\mathcal{U}_{p}^{1}$ and $\mathcal{I}_{p}^{1}$
for $p$ below the scale invariant threshold value; see \cite{Scholtes2011e}.\\

The paper is organised as follows: Section \ref{sectioncurvenergies} introduces integral curvature energies for 
arbitrary metric spaces, as this is no more complicated than doing so for arbitrary sets in $\R^{n}$ and even provides a simpler notation.
Then, in Section \ref{sectiondensityestimates}, we give lower bounds for the Hausdorff measure of annuli under
certain conditions on the Hausdorff density. We also introduce a new and slightly wider notion of Hausdorff density for set valued mappings.
In Section \ref{sectionapproximatetangents} we give some examples and simple properties of the different notions of tangents. 
Finally we are ready to prove the main theorem and compute the energies $1/\Delta$ \& $\mathcal{U}_{p}^{\he}$, $\mathcal{I}_{p}^{\he}$ and $\mathcal{M}_{p}^{\he}$ of the set $E$ in the Sections \ref{sectionUp}, \ref{sectionIp} 
and \ref{sectionMp}, respectively.
The topic of Section \ref{sectionresultsharp} is the proof of Proposition \ref{introductionweakapprtangentsoptimal}. To improve readability we have deferred several technical issues to the appendix.\\

\textbf{Acknowledgements}\\
The author wishes to thank his advisor Heiko von der Mosel for constant support and encouragement, reading and discussing the present paper, as well as giving many helpful remarks, 
like the idea to allow for $\he\not=1$ as in \cite{Lin2001a}. He is also thankful to Martin Meurer for the joint efforts that lead to Lemma \ref{finiteenergyfinitemeasureball}.
Furthermore the author is indebted to Thomas El Khatib, who gave some helpful remarks and a better proof for Lemma \ref{distanceintermsofangle}.

\section{Curvature energies and notation}\label{sectioncurvenergies}

For a set $X$ with outer measure $\V$ we write $\measurable(\V)$ for the \emph{$\V$ measurable sets} of $X$,
i.e. those sets $E$, which are measurable in the sense of Carath{\'e}ory:
\begin{align*}
	\V(M)=\V(M\cap E)+\V(M\backslash E)\quad\text{for all }M\subset X.
\end{align*}
Let $(X,\tau)$ be a topological space -- in this paper the topology is always induced by a metric -- then $\B(X)$ denotes the
\emph{Borel sets} of $(X,\tau)$. For two measurable spaces $(X,\mathcal{A})$ and $(Y,\mathcal{B})$ we say that a function
$f:(X,\mathcal{A})\to(Y,\mathcal{B})$ is \emph{$\mathcal{A}$--$\mathcal{B}$ measurable}, if $f^{-1}(B)\in \mathcal{A}$ for all $B\in\mathcal{B}$.
By $\HM^{\he}$ we denote the \emph{$\he$-dimensional Hausdorff measure} on a metric space $(X,d)$ and by
$\LM^{n}$ the \emph{$n$-dimensional Lebesgue measure} on $\R^{n}$. The \emph{extended real numbers} are indicated by the symbol $\overline\R$.\\

The thickness of a set was introduced by O. Gonzales and J. Maddocks in \cite{Gonzalez1999a}, where they also suggested to
investigate the integral curvature energies $\mathcal{U}_{p}^{1},\mathcal{I}_{p}^{1}$ and $\mathcal{M}_{p}^{1}$, which will be defined 
subsequently.

\begin{definition}[(Circumradius, interm. and global radius of curv., thickness)]
	Let $(X,d)$ be a metric space. We define the \emph{circumradius} of three distinct points $x,y,z\in X$ as the circumradius of the
	triangle defined by the, up to Euclidean motions unique, isometric embedding of these three points in the Euclidean plane, i.e.
	\begin{align}\label{formulacircumradius}
		\begin{split}
			\MoveEqLeft r:\{(x,y,z)\in X^{3}\mid d(x,y),d(y,z),d(z,x)>0\}=\vcentcolon D\to\overline\R,\\
			&(x,y,z)\mapsto\frac{abc}{\sqrt{(a+b+c)(a+b-c)(a-b+c)(-a+b+c)}},
		\end{split}
	\end{align}
	where $a\vcentcolon=d(x,y)$, $b\vcentcolon=d(y,z)$, $c\vcentcolon=d(z,x)$ and $\alpha/0=\infty$ for any $\alpha>0$.
	We also write $X_{0}\vcentcolon=X^{3}\backslash D$.
	Now we define the mappings $\rho:X^{2}\backslash \mathrm{diag}(X)\to \overline\R$ and $\rho_{G}:X\to\overline \R$ by
	\begin{align*}
		\rho(x,y)\vcentcolon=\inf_{\substack{w\in X\\x\not= y\not= w\not= x}}r(x,y,w)\quad\text{and}\quad
		\rho_{G}(x)\vcentcolon=\inf_{\substack{v,w\in X\\x\not= v\not= w\not= x}}r(x,v,w),
	\end{align*}
	which are often called \emph{intermediate} and \emph{global radius of curvature}, respectively. 
	Here $\mathrm{diag}(X)\vcentcolon=\{(x,x)\mid x\in X\}$ denotes the \emph{diagonal} of $X$.
	The \emph{thickness} is then defined to be
	\begin{align*}
		\Delta[X]\vcentcolon=\inf_{\substack{u,v,w\in X\\u\not= v\not= w\not= u}}r(u,v,w).
	\end{align*}
\end{definition}

\begin{remark}[(Different formulas for the circumradius)]\label{differentformulaskappa}
	We note that in $\R^{n}$ there are various formulas for the circumradius, for example one has the following 
	representations for $x,y,z\in\R^{n}$ mutually distinct
	\cite[(14) and (15), p.29]{Pajot2002a}
	\begin{align*}
		r(x,y,z)=\frac{\abs{x-y}}{2\abs{\sin(\measuredangle(x,z,y))}}
		=\frac{\abs{x-z}\abs{y-z}}{2\dist(z,L_{x,y})},
	\end{align*}
	where $L_{x,y}\vcentcolon=x+\R(x-y)$ is the straight line connecting $x$ and $y$. 
\end{remark}

\begin{lemma}[(Various curvature radii are upper semi-continuous)]\label{curvatureradiiareupersemi-continuous}
	Let $(X,d)$ be a metric space. Then
	\begin{align*}\begin{array}{rcll}
		r:			&X^{3}\backslash X_{0}			&\to\overline\R\qquad		&\text{is continuous},\\
		\rho:		&X^{2}\backslash\mathrm{diag}(X)		&\to\overline\R\qquad		&\text{is upper semi-continuous},\\
		\rho_{G}:	&X								&\to \overline \R\qquad	&\text{is upper semi-continuous}.
		\end{array}
	\end{align*}
\end{lemma}
\begin{proof}
	\textbf{Step 1}
		Let $((x_{n},y_{n},z_{n}))_{n\in\N}\subset D$ and $(x,y,z)\in D$ such that 
		$(x_{n},y_{n},z_{n})\to (x,y,z)$ in $X^{3}$ and set $f(x,y,z)\vcentcolon=(-a+b+c)(a-b+c)(a+b-c)$.\\
		\textbf{Case 1}
			Let us first	assume that $f(x,y,z)\not=0$. Then $f(x,y,z)>0$ and as $(x,y,z)\in D$ we have $r(x,y,z)<\infty$.
			Since $f$ is continuous, see Lemma \ref{metriccontinuous} we have $f(x_{n},y_{n},z_{n})\geq f(x,y,z)/2$ for $n$ large enough.
			Therefore $r(x_{n},y_{n},z_{n})\to r(x,y,z)$, because the numerator of (\ref{formulacircumradius}) is also continuous.\\
		\textbf{Case 2}
			If on the other hand $f(x,y,z)=0$, we have $f(x_{n},y_{n},z_{n})\to 0$ and 
			$g(x_{n},y_{n},z_{n})\vcentcolon= d(x_{n},y_{n})d(y_{n},z_{n})d(z_{n},x_{n})>g(x,y,z)/2$ for $n$ large enough, which gives us 
			$r(x_{n},y_{n},z_{n})\to r(x,y,z)=\infty$.\\\
	\textbf{Step 2}
		If we set $f_{z}:(x,y)\mapsto r(x,y,z)$ then according to the previous item the functions $f_{z}$ are upper semi-continuous and therefore, 
		see \cite[Remark 1.4 (ii), p.21]{Braides2002a}, also is
		\begin{align*}
			\rho(x,y)=\inf_{z\in X\backslash\{x,y\}}f_{z}(x,y).
		\end{align*}
	\textbf{Step 3}
		By arguing analogous to the proof of the preceding item we have that
		\begin{align*}
			\rho_{G}(x)=\inf_{y\in X\backslash \{x\}}\rho(x,y)
		\end{align*}
		is upper semi-continuous.
\end{proof}

\begin{lemma}[(Reciprocal radii of curvature are l.s.c. and measurable)]\label{reciprocalradiiofcurvaturearemeasurable}
	Let $(X,d)$ be a metric space. Then the functions
	\begin{align*}
		\begin{array}{cclrl}
		\kappa_{G}:&	 X&		\to\overline\R,&		
			\,x&\mapsto\frac{1}{\rho_{G}(x)},\\
		\kappa_{i}:&		X^{2}&	\to\overline\R,&		
			\,(x,y)&\mapsto\begin{cases}\frac{1}{\rho(x,y)},&(x,y)\in X^{2}\backslash \mathrm{diag}(X),\\0,&\text{else},\end{cases}\\
		\kappa:&		X^{3}&	\to\overline\R,&		
			\,(x,y,z)&\mapsto\begin{cases}\frac{1}{r(x,y,z)},&(x,y,z)\in X^{3}\backslash X_{0},\\0,&\text{else},\end{cases}
		\end{array}
	\end{align*}	
	with the convention $1/0=\infty$ and $1/\infty=0$ are lower semi-continuous and $\B(X)\mhyphen\B(\overline \R)$, $\B(X^{2})\mhyphen\B(\overline \R)$ and 
	$\B(X^{3})\mhyphen\B(\overline \R)$ measurable, respectively.
\end{lemma}
\begin{proof}
	Considering Lemmata \ref{curvatureradiiareupersemi-continuous} and \ref{reciprocalofsemi-continuousfunctions}
	the functions $\kappa_{G}$, $\kappa_{i}$ and $\kappa$ are lower semi-continuous on $X$, 
	$X^{2}\backslash \textrm{diag}(X)$ and $X^{3}\backslash X_{0}$ respectively. 
	This proves the proposition for $\kappa_{G}$. Now considering that the excluded sets $\mathrm{diag}(X)$ and $X_{0}$ are closed, 
	Lemma \ref{DeltaXandX0areclosed},
	and that the functions are non-negative on the whole space and $0$ on these sets, we know that they are lower semi-continuous on the entire space
	by Lemma \ref{extensionoflowersemicontinuousfunctions}.
	Now  Lemma \ref{semicontiniuousfunctionsaremeasurable} gives us Borel measurability.
\end{proof}

\begin{definition}[(A menagerie of integral curvature energies)]
	Let $(X,d)$ be a metric space and $\he,p\in (0,\infty)$. We are now able to define the following two-parameter families of 
	integral curvature energies 
	\begin{align*}
		\mathcal{U}_{p}^{\he}(X)&\vcentcolon=\int_{X}\kappa_{G}^{p}(x)\dHM(x),\\
		\mathcal{I}_{p}^{\he}(X)&\vcentcolon=\int_{X}\int_{X}\kappa_{i}^{p}(x,y)\dHM(x)\dHM(y),\\
		\mathcal{M}_{p}^{\he}(X)&\vcentcolon=\int_{X}\int_{X}\int_{X}\kappa^{p}(x,y,z)\dHM(x)\dHM(y)\dHM(z).
	\end{align*}
	The last of these energies, $\M_{p}^{\he}$, is often called \emph{$\he$-dimensional (integral) $p$-Menger curvature}.
\end{definition}

\begin{remark}[(Subtle differences in possible definitions of energies)]\label{remarksubtledifferences}
	We want to remark that in the Euclidean case the measure in the integrals is the Hausdorff measure on the set $X$ (respective to
	the subspace metric, i.e. the restriction of the metric of $\R^{n}$ to the set $X$), in contrast to the Hausdorff measure on $\R^{n}$.
	As we shall see shortly this enables us to include non-measurable sets, contrary to the other approach, where the energy might
	not exist on non-measurable sets, which can easily be seen by the example of a Vitali type set on the unit circle. We suspect that the gain of permitted sets
	when comparing \cite{Hahlomaa2008a} for $\R^{n}$ to \cite{Leger1999a}, where only Borel sets were permitted, might be related to this
	matter.
\end{remark}

We shall now be concerned with the existence of these integral curvature energies, which is why we first take a closer look at the integrands.

\begin{lemma}[(Various integrand functions are l.s.c. and measurable)]\label{variousintegrandfuntionsarelsc}
	Let $(X,d)$ be a metric space. Then for all $p\in (0,\infty)$ the following functions
	\begin{align*}
		y&\mapsto \int_{X}\kappa_{i}^{p}(x,y)\dd\HM^{\he}(x)\\
		y&\mapsto \int_{X}\kappa^{p}(x,y,z)\dd\HM^{\he}(x)\quad\text{for all }z\in X\\
		z&\mapsto \int_{X}\int_{X}\kappa^{p}(x,y,z)\dd\HM^{\he}(x)\dd\HM^{\he}(y)
	\end{align*}
	are lower semi-continuous and $\B(X)\mhyphen\B(\overline\R)$ measurable.
\end{lemma}
\begin{proof}
	\textbf{Step 1}
		By Lemma \ref{reciprocalradiiofcurvaturearemeasurable}, $\kappa\geq 0$ and  Lemma \ref{powersofsemicontiniuous}
		we know that $\kappa^{p}$ is lower semi-continuous. 
		Let $a_{n}\to a$ in $X$. As for fixed $x,y,z\in X$ we have $(a_{n},y,z)\to (a,y,z)$ and therefore
		\begin{align*}
			\kappa^{p}(a,y,z)\leq\liminf_{n\to\infty} \kappa^{p}(a_{n},y,z),
		\end{align*}
		so that $\kappa^{p}(\cdot,y,z)$, $\kappa^{p}(x,\cdot,z)$ and $\kappa^{p}(x,y,\cdot)$ are lower semi-continuous 
		and hence $\B(X)$--$\B(\overline\R)$ measurable,
		see Lemma \ref{semicontiniuousfunctionsaremeasurable}. Using Fatou's Lemma \cite[Theorem 1, p.19]{Evans1992a} we obtain
		\begin{align*}
			\int_{X}\kappa^{p}(x,y,a)\dHM(x)\leq \int_{X} \liminf_{n\to\infty}\kappa^{p}(x,y,a_{n})\dHM(x)
			\leq \liminf_{n\to\infty} \int_{X} \kappa^{p}(x,y,a_{n})\dHM(x).
		\end{align*}
		This tells us that for fixed $x,y,z$ the mappings $\int_{X}\kappa^{p}(x,\cdot,z)\dHM(x)$ and $\int_{X}\kappa^{p}(x,y,\cdot)\dHM(x)$ are 
		lower semi-continuous and hence measurable.\\
	\textbf{Step 2}
		Let $z_{n}\to z$ in $X$. If we use Fatou's Lemma and integrate again, we obtain
		\begin{align*}
			\MoveEqLeft\int_{X}\int_{X}\kappa^{p}(x,y,z)\dHM(x)\dHM(y)
			\leq \int_{X}\liminf_{n\to\infty} \int_{X} \kappa^{p}(x,y,z_{n})\dHM(x)\dHM(y)\\
			&\leq \liminf_{n\to\infty}\int_{X} \int_{X} \kappa^{p}(x,y,z_{n})\dHM(x)\dHM(y),
		\end{align*}
		so that $z\mapsto \int_{X}\int_{X}\kappa^{p}(x,y,z)\dHM(x)\dHM(y)$ is lower semi-continuous and hence measurable. 
		For the function involving $\kappa_{i}$ we argue analogously.
\end{proof}

\begin{lemma}[(Integral curvature energies are well-defined)]\label{integralcurvatureenergiesarewelldefinedinms}
	Let $(X,d)$ be a metric space. Then for all $\he,p\in (0,\infty)$ the curvature energies $\mathcal{U}_{p}^{\he}(X)$, 
	$\mathcal{I}_{p}^{\he}(X)$ and $\M_{p}^{\he}(X)$ are well defined.
\end{lemma}
\begin{proof}
	This is a simple consequence of Lemma \ref{reciprocalradiiofcurvaturearemeasurable} and Lemma \ref{variousintegrandfuntionsarelsc}
	together with the fact that the integrands are non-negative, see \cite[Remark, p.18]{Evans1992a}.
\end{proof}

\begin{lemma}[(Inequality between integral curvature energies)]\label{inequalityforcurvatureenergies}
	Let $(X,d)$ be a metric space with $\HM^{\he}(X)<\infty$ and $\he,p\in(0,\infty)$, then
	\begin{align*}
		\M_{p}^{\he}(X)\leq\HM^{\he}(X)\mathcal{I}_{p}^{\he}(X)\leq\HM^{\he}(X)^{2}\mathcal{U}_{p}^{\he}(X)\leq\frac{\HM^{\he}(X)^{3}}{\Delta[X]^{p}}.
	\end{align*}
\end{lemma}
\begin{proof}
	Clearly for all distinct $x,y,z\in X$ we have
	\begin{align*}
		\Delta[X]\leq\rho_{G}(x)\leq\rho(x,y)\leq r(x,y,z),
	\end{align*}
	which gives us 
	\begin{align}\label{inequalitycRDelta}
		\kappa(x,y,z)\leq \kappa_{i}(x,y)\leq \kappa_{G}(x)\leq\frac{1}{\Delta[X]}\quad\text{for all }x,y,z\in X
	\end{align}
	and thus the proposition.
\end{proof}

By successively using the H{\"o}lder inequality from the inner to the outer integral one can easily prove

\begin{lemma}[(Comparison of curvature energies for different $p$)]\label{comparisonofcurvatureenergiesfordifferentp}
	Let $(X,d)$ be a metric space with $\HM^{\he}(X)<\infty$, $\he\in (0,\infty)$ and $0<p<q <\infty$. Then
	\begin{align*}
		\mathcal{U}_{p}^{\he}(X)&\leq \HM^{\he}(X)^{(1-p/q)}\mathcal{U}_{q}^{\he}(X)^{p/q},\\
		\mathcal{I}_{p}(X)^{\he}&\leq \HM^{\he}(X)^{2(1-p/q)}\mathcal{I}_{q}^{\he}(X)^{p/q},\\
		\M_{p}^{\he}(X)&\leq \HM^{\he}(X)^{3(1-p/q)}\M_{q}^{\he}(X)^{p/q}.
	\end{align*}
\end{lemma}
\begin{proof}
	For $a=q/p>1$ and $b=q/(q-p)$ we obtain
	\begin{align*}
		\MoveEqLeft\M_{p}^{\he}(X)=\int_{X}\int_{X}\int_{X}\kappa^{p}(x,y,z)\dHM(x)\dHM(y)\dHM(z)\\
		&\leq \int_{X}\int_{X}\HM^{\he}(X)^{1/b}\Big(\int_{X}\kappa^{pa}(x,y,z)\dHM(x)\Big)^{1/a}\dHM(y)\dHM(z)\\
		&\leq \HM^{\he}(X)^{1/b}\int_{X} \HM^{\he}(X)^{1/b} \Big[\int_{X}\Big(\int_{X}\kappa^{q}(x,y,z)\dHM(x)\Big)^{a\cdot 1/a}\dHM(y)\Big]^{1/a}\dHM(z)\\
		&\leq \HM^{\he}(X)^{3/b}\Big(\int_{X}\Big[\int_{X}\int_{X}\kappa^{q}(x,y,z)\dHM(x)\dHM(y)\Big]^{a\cdot 1/a}\dHM(z)\Big)^{1/a}\\
		&\leq \HM^{\he}(X)^{3(1-p/q)}\Big(\int_{X}\int_{X}\int_{X}\kappa^{q}(x,y,z)\dHM(x)\dHM(y)\dHM(z)\Big)^{p/q}.
	\end{align*}
	The inequalities for the other two energies are proven analogously.
\end{proof}

Later on we often use the contrapositive of the following lemma to show that a set has infinite curvature energy.

\begin{lemma}[($\F(B_{r})\to 0$ if $\F(X)<\infty$)]\label{nicebehaviourlimitofsmallballs}
	Let $(X,d)$ be a metric space with $\HM^{\he}(X)<\infty$, $\he,p\in (0,\infty)$, $\F\in \{\mathcal{U}_{p}^{\he},\mathcal{I}_{p}^{\he},\mathcal{M}_{p}^{\he}\}$. 
	If we have finite energy $\F(X)<\infty$ then for all $x\in X$
	\begin{align*}
		\lim_{r\downarrow 0}\F(B_{r}(x))=0.
	\end{align*}
\end{lemma}
\begin{proof}
	Let $x_{0}\in X$ and assume that there is a monotonically decreasing sequence $(r_{n})_{n\in\N}$, $r_{n}>0$ with $\lim_{n\to\infty}r_{n}=0$, 
	such that
	$\F(B_{r_{n}}(x_{0}))\geq c>0$ for all $n\in\N$. We first note that as $B_{r}(x_{0})\in\measurable(\HM^{\he})$ and measures 
	are continuous on monotonically decreasing sets $E_{j}$, 
	if $E_{1}$ has finite measure, \cite[Theorem 1.1, (b), p.2]{Falconer1985a} we have
	\begin{align*}
		\lim_{n\to\infty}\HM^{\he}(B_{r_{n}}(x_{0}))=\HM^{\he}(\lim_{n\to\infty}B_{r_{n}}(x_{0}))=\HM^{\he}(\{x_{0}\})=0.
	\end{align*}
	Let 
	\begin{align*}
		f\in\Big\{x\mapsto \kappa_{G}^{p}(x),
		y\mapsto \int_{X}\kappa_{i}^{p}(x,y)\dd\HM^{\he}(x), 
		z\mapsto \int_{X}\int_{X}\kappa^{p}(x,y,z)\dd\HM^{\he}(x)\dd\HM^{\he}(y)\Big\}
	\end{align*}
	be the corresponding integrand to $\F$.
	Then $f$ is measurable, as we have seen in
	Lemma \ref{reciprocalradiiofcurvaturearemeasurable} and Lemma \ref{variousintegrandfuntionsarelsc}, and
	\begin{align*}
		\int_{B_{r_{n}}(x_{0})}f\dd\HM^{\he}\geq \F(B_{r_{n}}(x_{0}))\geq c>0.
	\end{align*}
	To conclude the proof we employ Lemma \ref{conditionforinfiniteintegral} for the different integrands $f$ and obtain 
	the desired contradiction, namely
	$\F(X)=\int_{X}f\dd\HM^{\he}=\infty$.
\end{proof}

\begin{lemma}[(Condition for infinite integral)]\label{conditionforinfiniteintegral}
	Let $\V$ be a regular outer measure on $X$, $f:(X,\measurable(V))\to (\overline\R,\B(\overline\R))$, 
	$f\geq 0$ measurable and $X_{n+1}\subset X_{n}\subset X$, $X_{n}\in \measurable(\V)$ 
	for $n\in\N$, such that $\V(X_{n})\to 0$. If
	\begin{align*}
		\int_{X_{n}}f\dd\V\geq c>0\quad\text{ for all  }n\in\N
	\end{align*}
	then
	\begin{align*}
		\int_{X}f\dd\V=\infty.
	\end{align*}
\end{lemma}
\begin{proof}
	We prove the contrapositive. Let $\int_{X}f\dd\V<\infty$. Then $f(1-\chi_{X_{n}})$ are measurable \cite[2.6 Proposition, p.45]{Folland1999a} and converge pointwise and monotonically increasing to $f$, 
	so that by the monotone convergence theorem, see for example
	\cite[1.3, Theorem 2, p.20]{Evans1992a} we have
	\begin{align*}
		\int_{X}f(1-\chi_{X_{n}})\dd\V=\int_{X}f\dd\V-\int_{X_{n}}f\dd\V\rightarrow \int_{X}f\dd\V
	\end{align*}
	and hence the proposition.
\end{proof}

We also need the following

\begin{lemma}[(Decomposition of triple integral)]\label{decompositionoftripleintegral}
	Let $\V$ be an outer measure on $X$ and $X_{i}\in\measurable(\V)$, $i\in\N$ with $\V(X_{i}\cap X_{j})=0$ for $i\not=j$ and 
	$X=\bigcup_{i\in\N}X_{i}$. 
	Let $f:X^{3}\to\overline\R$, $f\geq 0$ be such that for all $x,y,z\in X$ the mappings
	\begin{align*}
		x\mapsto f(x,y,z), \quad y\mapsto\int_{X}f(x,y,z)\dd\V(x)\quad \text{and}\quad z\mapsto\int_{X}\int_{X}f(x,y,z)\dd\V(x)\dd\V(y)
	\end{align*}
	are $\measurable(\V)$--$\B(\overline\R)$ measurable.
	Then
	\begin{align*}
		\int_{X}\int_{X}\int_{X}f(x,y,z)\dd\V(x)\dd\V(y)\dd\V(z)=\sum_{i,j,k\in\N}\int_{X_{k}}\int_{X_{j}}\int_{X_{i}}f(x,y,z)\dd\V(x)\dd\V(y)\dd\V(z).
	\end{align*}
\end{lemma}
\begin{proof}
	This is a repeated application of the monotone convergence theorem. If $g:X\to\overline \R$, $g\geq 0$ is $\measurable(\V)$--$\B(\overline \R)$ 
	measurable, then
	so are $g_{n}\vcentcolon=\sum_{i=1}^{n}g\cdot\chi_{X_{i}}$ and $g_{n}\to g$ monotonically. Hence the monotone convergence theorem gives us
	\begin{align*}
		\sum_{i\in \N}\int_{X_{i}}g\dd\V=\lim_{n\to\infty}\int_{X}g_{n}\dd\V=\int_{X}g\dd\V.
	\end{align*}
\end{proof}

Just after the first version of this paper had been written up Martin Meurer, who also did a higher dimensional version of this,
and the author could show the following lemma.
It offers us the opportunity to include sets with infinite measure in our
subsequent theorems.

\begin{lemma}[(Finite energy implies finite measure on all balls)]\label{finiteenergyfinitemeasureball}
	Let $\he\in[1,\infty)$, $p\in (0,\infty)$, $\F\in\{\mathcal{U}_{p}^{\he},\mathcal{I}_{p}^{\he},\mathcal{M}_{p}^{\he}\}$ and $X\subset \R^{n}$ be a set with $\mathcal{F}(X)<\infty$.
	Then for all $x\in \R^{n}$ and all $R>0$ we have $\HM^{\he}(X\cap \overline B_{R}(x))<\infty$.
\end{lemma}
\begin{proof}
	We argue by contradiction and therefore assume that this is not the case.\\
	\textbf{Step 1}
		We show that there is an $x_{0}\in \overline B_{R}(x)$ with
		\begin{align}\label{infinitemeasureonsmallballs}
			\HM^{\he}(X\cap B_{r}(x_{0}))=\infty\quad\text{for all }r>0.
		\end{align}
		According to our assumption there exists $x\in\R^{n}$ and $R>0$, such that $\HM^{\he}(X\cap B_{R}(x))=\infty$. By a
		covering argument we know that for any $n\in\N$ there is an $x_{n}\in B_{R}(x)$, such that $\HM^{\he}(X\cap B_{1/n}(x_{n}))=\infty$. 
		As $\overline B_{R}(x)$ is compact, there is a subsequence, such that $x_{n_{k}}\to x_{0}\in \overline B_{R}(x)$. 
		Then $\HM^{\he}(X\cap B_{r}(x_{0}))=\infty$ for all $r>0$, because 
		\begin{align*}
			\sup_{y\in B_{1/n_{k}}(x_{n_{k}})}d(x_{0},y)\leq d(x_{0},x_{n_{k}})+\frac{1}{n_{k}}\to 0.
		\end{align*}
	\textbf{Step 2}
		For $\rho>0$ we can find $r=r(\rho)$, $A\vcentcolon=[B_{\rho}(x_{0})\backslash B_{r}(x_{0})]$ such that
		$\HM^{\he}(X\cap A )\geq 3\rho$, because $B_{\rho}(x_{0})\backslash B_{r}(x_{0})\in\measurable(\HM^{\he}_{X})$ and by Lemma
		\ref{subspacehausdorffmeasure} and the continuity of measures on increasing sets \cite[Theorem 1.1, (a), p.2]{Falconer1985a} we have
		\begin{align*}
			\MoveEqLeft
			\HM^{\he}(X\cap B_{\rho}(x_{0}))=\HM^{\he}_{X}(B_{\rho}(x_{0}))=\HM^{\he}_{X}(B_{\rho}(x_{0})\backslash\{x_{0}\})\\
			&=\HM^{\he}_{X}(\bigcup_{n\in\N}B_{\rho}(x_{0})\backslash B_{1/n}(x_{0}))
			=\lim_{n\to\infty}\HM^{\he}_{X}(B_{\rho}(x_{0})\backslash B_{1/n}(x_{0}))=\infty. 
		\end{align*}
		Then there exists a direction $s\in\mathbb{S}^{n-1}$ and an $\epsilon>0$, such that
		\begin{align}\label{bothmeasurespositive}
			\HM^{\he}(X\cap A\cap C_{s,\epsilon}(x_{0}))>0\quad\text{and}\quad\HM^{\he}([X\cap A]\backslash C_{s,2\epsilon}(x_{0}))>0,
		\end{align}
		because, by a covering and compactness argument similar to that of Step 1, there is a direction $s$,
		such that for all $\epsilon>0$ we have $\HM^{\he}(X\cap A\cap C_{s,\epsilon}(x_{0}))>0$.
		If we assume that $\HM^{\he}([X\cap A]\backslash C_{s,2\epsilon}(x_{0}))=0$ for all $\epsilon>0$, we obtain a contradiction for
		$N_{n}\vcentcolon=[X\cap A]\backslash C_{s,1/n}(x_{0})$ as
		\begin{align*} 
			\HM^{\he}([X\cap A]\backslash L)=\HM^{\he}(\bigcup_{n\in\N}N_{n})\leq\sum_{n\in\N}\HM^{\he}(N_{n})=0
		\end{align*}	
		by
		\begin{align*}
			3\rho\leq \HM^{\he}(X\cap A)=\HM^{\he}([X\cap A]\backslash L)+\HM^{\he}(X\cap A\cap L)=\HM^{\he}(X\cap A\cap L)\leq 2\rho,
		\end{align*}
		where $L=x_{0}+[-\rho,\rho]s$. For the last inequality we needed $\he\in [1,\infty)$.\\
	\textbf{Step 3}
		Denote $C\vcentcolon=X\cap A\cap C_{s,\epsilon}(x_{0})$ and $C'\vcentcolon=[X\cap A]\backslash C_{s,2\epsilon}(x_{0})$
		the sets from (\ref{bothmeasurespositive}). By Lemma \ref{distanceintermsofangle} we have
		$\dist(L_{x,y},x_{0})\geq \sin(\epsilon)r/2$ for all $x\in C$ and all $y\in C'$, so that for all $z\in B_{\sin(\epsilon)r/4}(x_{0})$ we have
		\begin{align*}
			\dist(L_{x,y},z)\geq \dist(L_{x,y},x_{0})-d(z,x_{0})\geq \sin(\epsilon)r/4
		\end{align*}
		and hence by (\ref{infinitemeasureonsmallballs})
		\begin{align*}
			\MoveEqLeft
			\M_{p}^{\he}(X)
			\geq \int_{C}\int_{C'}\int_{B_{\sin(\epsilon)r/4}(x_{0})}\frac{[\sin(\epsilon)r/4]^{p}}{r^{2p}}\dHM(z)\dHM(y)\dHM(x)\\
			&\geq \HM^{\he}(C)\HM^{\he}(C')\HM^{\he}(B_{\sin(\epsilon)r/4}(x_{0}))\frac{[\sin(\epsilon)r/4]^{p}}{r^{2p}}
			\stackrel{\text{(\ref{infinitemeasureonsmallballs})}}{=}\infty.
		\end{align*}
		With a similar argument for the other energies we have proven the proposition.
\end{proof}

\begin{corollary}[(Finite energy implies that $\HM^{\he}$ is a Radon measure)]
	Let $\he\in [1,\infty)$, $p\in (0,\infty)$, $\F\in\{\mathcal{U}_{p}^{\he},\mathcal{I}_{p}^{\he},\mathcal{M}_{p}^{\he}\}$ and $X\subset \R^{n}$ be a set with $\mathcal{F}(X)<\infty$.
	Then $\HM^{\he}_{X}$ is a Radon measure.
\end{corollary}
\begin{proof}
	This is a direct consequence of Lemma \ref{finiteenergyfinitemeasureball}.
\end{proof}

\begin{lemma}[(Consequences of finite energy for $\he\in (0,1)$)]\label{consequencesfiniteenergyhe<1}
	Let $\he\in(0,1)$, $p\in (0,\infty)$, $\F\in\{\mathcal{U}_{p}^{\he},\mathcal{I}_{p}^{\he},\mathcal{M}_{p}^{\he}\}$ and $X\subset \R^{n}$ be a set with $\mathcal{F}(X)<\infty$.
	For all $x_{0}\in X$ we have one of the following propositions
	\begin{itemize}
		\item
			there is an $\rho>0$, such that $\HM^{\he}(X\cap \overline B_{r}(x_{0}))<\infty$ for all $r\in (0,\rho)$, or
		\item
			there is a direction $s\in\mathbb{S}^{n-1}$, such that $\Theta^{\he}(\HM^{\he},X\backslash C_{s,\epsilon}(x_{0}),x_{0})=0$ for all $\epsilon>0$.
	\end{itemize}
\end{lemma}
\begin{proof}
	Assume that both alternatives are false, i.e. that 
	\begin{align}\label{measureinballssequenceisinfinite}
		\HM^{\he}(X\cap \overline B_{r_{n}}(x_{0}))=\infty\text{ for a sequence }r_{n}\downarrow 0
	\end{align}
	and for all $s\in \mathbb{S}^{n-1}$ there is $\epsilon_{s}>0$, such that 
	\begin{align}\label{densityonarbitrarylinepositive}
		\Theta^{*\he}(\HM^{\he},X\backslash C_{s,\epsilon_{s}}(x_{0}),x_{0})>0.
	\end{align}
	We show that then $\M_{p}^{\he}(X)=\infty$. For the other energies a similar argument can be applied.
	Denote $A_{r}\vcentcolon=\overline B_{\rho}(x_{0})\backslash B_{r}(x_{0})$ and $L_{s}=x_{0}+\R s$ for a direction $s\in\mathbb{S}^{n-1}$.
	Due to (\ref{measureinballssequenceisinfinite}) and an argument similar to that indicated in Step 2 of Lemma \ref{finiteenergyfinitemeasureball}, we can find a direction 
	$s_{0}\in \mathbb{S}^{n-1}$, such that $\HM^{\he}([X\cap A_{r}]\cap C_{s_{0},\epsilon}(x_{0}))>0$ for all $\epsilon>0$ and all $\rho>0$, if $r\in (0,\rho)$ is small enough.\\
	\textbf{Case 1}
		We first investigate the case that $\HM^{\he}(X\cap \overline B_{r_{n}}(x_{0})\cap L_{s_{0}})<\infty$ for $n\geq N$.
		Now we can argue analogously to Step 2 from Lemma \ref{finiteenergyfinitemeasureball} to obtain for all $\rho\in(0,r_{N})$
		a contradiction to $\HM^{\he}([X\cap A_{r}]\backslash C_{s_{0},2\epsilon}(x_{0}))=0$ for all $\epsilon>0$ and all $r\in (0,\rho)$ by
		\begin{align*}
			\infty\stackrel{\text{(\ref{measureinballssequenceisinfinite})}}{=}\lim_{r\downarrow 0}\HM^{\he}(X\cap A_{r})=\HM^{\he}(X\cap \overline B_{\rho}(x_{0})\cap L_{s_{0}})
			\leq \HM^{\he}(X\cap \overline B_{r_{N}}(x_{0})\cap L_{s_{0}})<\infty.
		\end{align*}
		Therefore we have shown the analogous result to Step 2 from Lemma \ref{finiteenergyfinitemeasureball} and can
		use Step 3 of this lemma to obtain $\M_{p}^{\he}(X)=\infty$.\\
	\textbf{Case 2}
		It is left to deal with the case that there is a subsequence, such that $\HM^{\he}(X\cap \overline B_{r_{n_{k}}}(x_{0})\cap L_{s_{0}})=\infty$ for $k\in\N$.
		Now we can use (\ref{densityonarbitrarylinepositive}) to obtain $\HM^{\he}(X\cap B_{r_{n_{k}}}(x_{0})\backslash C_{s_{0},\epsilon_{s_{0}}}(x_{0}))>0$. Then we argue again as in
		Step 3 of Lemma \ref{finiteenergyfinitemeasureball}, using (\ref{measureinballssequenceisinfinite}), to obtain $\M_{p}^{\he}(X)=\infty$.
\end{proof}

\section{Hausdorff density and lower estimates of annuli}\label{sectiondensityestimates}

In this section we remind the reader of the definition of Hausdorff density, introduce a slightly wider notion for set valued mappings and prove some properties of these densities.
More importantly we estimate the Hausdorff measure of annuli from below under the assumption that the densities fulfill certain conditions.

\begin{definition}[(Hausdorff density for set-valued mappings)]
	Let $(X,d)$ be a metric space, $x\in X$, $\he\in (0,\infty)$ and $A:(0,\rho)\to \textrm{Pot}(X)$. Then
	\begin{align*}
		\Theta_{*}^{\he}(\HM^{\he},A(r),x)&\vcentcolon=\liminf_{r\downarrow 0}\frac{\HM^{\he}(A(r)\cap \overline B_{r}(x))}{(2r)^{\he}},\\
		\Theta^{*\he}(\HM^{\he},A(r),x)&\vcentcolon=\limsup_{r\downarrow 0}\frac{\HM^{\he}(A(r)\cap \overline B_{r}(x))}{(2r)^{\he}}
	\end{align*}
	are called the \emph{lower and upper $\he$-dimensional Hausdorff density of $A$ in $x$}.
	If upper and lower density coincide we call their common value \emph{Hausdorff density}
	and denote it by $\Theta^{\he}(\HM^{\he},A(r),x)$. Here
	\begin{align*}
		\overline B_{r}(x)\vcentcolon=\{y\in X\mid d(x,y)\leq r\}
	\end{align*}
	is the closed ball of radius $r$ about $x$. If $A(r)\equiv A$ is constant we will usually identify the mapping with the constant and neglect the argument.
\end{definition}

\begin{remark}[(Warning: closure of ball $\mathrm{cl}(B_{r}(x))$ may not equal closed ball $\overline B_{r}(x)$)]
	In normed vector spaces the notion of closed balls and the closure of balls coincides. However, in metric spaces this may not be the case, 
	as can be quickly seen
	by looking at $B_{1}(0)=\{0\}$, $\mathrm{cl}(B_{1}(0))=\{0\}$ and $\overline B_{1}(0)=\R$ in $(\R,d)$, where $d$ is the discrete metric.
\end{remark}

\begin{lemma}[(Implications of positive lower density)]\label{implicationspositivelowerdensity}
	Let $(X,d)$ be a metric space, $x\in X$, $A:(0,\varrho)\to\textrm{Pot}(X)$, $\he\in (0,\infty)$ and $\vartheta_{*}\vcentcolon=\Theta_{*}^{\he}(\HM^{\he},A(r),x)>0$. 
	Then for all $\theta\in (0,2^{\he}\vartheta_{*})$ there is $\rho>0$, such that for all $r\in (0,\rho)$ we have
	\begin{align*}
		\theta r^{\he}\leq \HM^{\he}(A(r)\cap \overline B_{r}(x)).
	\end{align*}
\end{lemma}
\begin{proof}
	Fix $\theta\in (0,2^{\he}\vartheta_{*})$ and assume that the proposition if false. Then for all $\rho>0$ there is $r_{\rho}\in (0,\rho)$, such that
	\begin{align}\label{conditionzerolowerdensity}
		\HM^{\he}(A(r_{\rho})\cap \overline B_{r_{\rho}}(x))<\theta r_{\rho}^{\he}.
	\end{align}
	Choose $\rho_{n}=n^{-1}$ and obtain a sequence $r_{n^{-1}}$, such that $r_{n^{-1}}\to 0$ and (\ref{conditionzerolowerdensity}), but this means that
	$\vartheta_{*}=\Theta_{*}^{\he}(\HM^{\he},A(r),x)\leq \theta/2^{\he}$, which contradicts $\theta\in (0,2^{\he}\vartheta_{*})$.
\end{proof}

\begin{lemma}[(Implications of finite upper density)]\label{implicationsfiniteupperdensity}
	Let $(X,d)$ be a metric space, $x\in X$, $A:(0,\varrho)\to\textrm{Pot}(X)$, $\he\in (0,\infty)$ and $\vartheta^{*}\vcentcolon=\Theta^{*\he}(\HM^{\he},A(r),x)<\infty$. 
	Then for all $\theta\in (2^{\he}\vartheta^{*},\infty)$, there is $\rho>0$, such that for all $r\in (0,\rho)$ we have
	\begin{align*}
		 \HM^{\he}(A(r)\cap \overline B_{r}(x))\leq \theta r^{\he}.
	\end{align*}
\end{lemma}
\begin{proof}
	Fix $\theta\in (2^{\he}\vartheta^{*},\infty)$ and assume that the proposition if false. Then for all $\rho>0$ there is $r_{\rho}\in (0,\rho)$, such that
	\begin{align}\label{conditionlargeupperdensity}
		\theta r_{\rho}^{\he}<\HM^{\he}(A(r_{\rho})\cap \overline B_{r_{\rho}}(x)).
	\end{align}
	Choose $\rho_{n}=n^{-1}$ and obtain a sequence $r_{n^{-1}}$, such that $r_{n^{-1}}\to 0$ and (\ref{conditionlargeupperdensity}), but this means that
	$\theta/2^{\he}\leq \Theta^{*\he}(\HM^{\he},A(r),x)=\vartheta^{*}$, which contradicts $\theta\in (2^{\he}\vartheta^{*},\infty)$.
\end{proof}

\begin{lemma}[(Simultaneous estimate of annuli)]\label{simultaneousestimateannuli}
	Let $(X,d)$ be a metric space, $\he\in (0,\infty)$, $A,B:(0,\rho)\to\textrm{Pot}(X)$, $x\in X$ with
	\begin{align*}
		0<\Theta_{*}^{\he}(\HM^{\he},A(r),x),\quad 0<\Theta^{*\he}(\HM^{\he},B(r),x)\quad\text{and}\quad\Theta^{*\he}(\HM^{\he},X,x)<\infty.
	\end{align*}
	Then there exists a $q_{0}\in (0,1)$, a sequence $(r_{n})_{n\in\N}$, $r_{n}>0$, $\lim_{n\to\infty}r_{n}=0$ and a constant $c>0$ such that 
	\begin{align*}
		c r_{n}^{\he}\leq \min\{\HM^{\he}(A(r_{n})\cap[\overline B_{r_{n}}(x)\backslash B_{q_{0}r_{n}}(x)]),
		\HM^{\he}(B(r_{n})\cap [\overline B_{r_{n}}(x)\backslash B_{q_{0}r_{n}}(x)])\}.
	\end{align*}	
\end{lemma}
\begin{proof}
	\textbf{Step 1}
		By our hypothesis $\Theta^{*\he}(\HM^{\he},B(r),x)=\delta_{0}>0$ and $\Theta^{*\he}(\HM^{\he},X,x)=\vcentcolon\theta/4^{\he}<\infty$ there are $r_{n}>0$, $r_{n}\to 0$, such that
		\begin{align*}
			\delta_{0}r_{n}^{\he}\leq \HM^{\he}(B(r_{n})\cap \overline B_{r_{n}}(x))
		\end{align*}
		and
		\begin{align*}
			\HM^{\he}(B(r_{n})\cap \overline B_{qr_{n}}(x))\leq \HM^{\he}(\overline B_{qr_{n}}(x))\leq \theta q^{\he}r_{n}^{\he}\quad\text{for all }q\in (0,1),
		\end{align*}
		see Lemma \ref{implicationsfiniteupperdensity}. Together this means that
		\begin{align*}
			\MoveEqLeft
			\HM^{\he}(B(r_{n})\cap [\overline B_{r_{n}}(x)\backslash B_{qr_{n}}(x)])\\
			&\geq \HM^{\he}(B(r_{n})\cap \overline B_{r_{n}}(x))-\HM^{\he}(B(r_{n})\cap B_{qr_{n}}(x))\\
			&\geq \HM^{\he}(B(r_{n})\cap \overline B_{r_{n}}(x))-\HM^{\he}(B(r_{n})\cap \overline B_{qr_{n}}(x))\\
			&\geq (\delta_{0}-\theta q^{\he})r_{n}^{\he}\geq \delta_{0} r_{n}^{\he}/2,
		\end{align*}
		if we choose $q^{\alpha}\leq \delta_{0}/(2\theta)<1$.\\
	\textbf{Step 2}
		As $0<\delta_{1}\vcentcolon= \Theta_{*}^{\he}(\HM^{\he},A(r),x)$ we know that
		\begin{align*}
			\delta_{1}r_{n}^{\he}\leq \HM^{\he}(A(r_{n})\cap \overline B_{r_{n}}(x))
		\end{align*}
		and can use the argument from Step 1 to obtain
		\begin{align*}
			\HM^{\he}(A(r_{n})\cap [\overline B_{r_{n}}(x)\backslash B_{qr_{n}}(x)])\geq (\delta_{1}-\theta q^{\he})r_{n}^{\he}\geq \delta_{1} r_{n}^{\he}/2
		\end{align*}
		if we choose $q^{\he}\leq \delta_{1}/(2\theta)<1$.\\
	\textbf{Step 3}
		Combining the results from the previous steps we obtain the proposition for $q_{0}=[\min\{\delta_{1},\delta_{2}\}/(2\theta)]^{1/\he}\in (0,1)$ and $c=\min\{\delta_{1},\delta_{2}\}/2$.
\end{proof}

\begin{lemma}[(Existence of positive upper density in finite decomposition)]\label{positiveupperdensityinfinitedecomposition}
	Let $(X,d)$ be a metric space $x\in X$, $\he\in (0,\infty)$, $\Theta^{*\he}(\HM^{\he},X,x)>0$ and $X_{i}\subset X$, $i\in\{1,\ldots,N\}$ 
	such that $X=\bigcup_{i=1}^{N}X_{i}$.
	Then there exists an $n\in\{1,\ldots,N\}$, such that $\Theta^{*\he}(\HM^{\he},X_{n},x)>0$.
\end{lemma}
\begin{proof}
	Assume that this is not the case. Then we obtain a contradiction 
	to $\Theta^{*\he}(\HM^{\he},X,x)>0$ by
	\begin{align*}
		\lim_{n\to\infty}\frac{\HM^{\he}(X\cap \overline B_{r_{n}}(x))}{(2r_{n})^{\he}}\leq \lim_{n\to\infty}\sum_{i=1}^{N} 
		\frac{\HM^{\he}(X_{i}\cap \overline B_{r_{n}}(x))}{(2r_{n})^{\he}}=0,
	\end{align*}
	for any sequence of radii $(r_{n})_{n\in\N}$, $r_{n}>0$, $\lim_{n\to\infty}r_{n}=0$.
\end{proof}

\begin{remark}[(Lemma \ref{positiveupperdensityinfinitedecomposition} is not true for countable decomposition)]
	If we choose $X=[0,1]$, $X_{0}=\{0\}$ and $X_{n}=(2^{-n},2^{-n+1}]$, we see that $\Theta^{*1}(\HM^{1},X_{n},0)=0$ for all $n\in\N_{0}$,
	but $\Theta^{*1}(\HM^{1},X,0)=1/2>0$.
\end{remark}

\begin{remark}[(In $\R^{n}$ we do not need $x\in X$)]\label{positiveupperdensityinfinitedecompositionremark}
	Note that for example in case $X\subset \R^{n}$ we do not require $x\in X$ in Lemma \ref{implicationspositivelowerdensity}, Lemma \ref{implicationsfiniteupperdensity}
	Lemma \ref{simultaneousestimateannuli} and Lemma \ref{positiveupperdensityinfinitedecomposition}.
\end{remark}

We would like to remind the reader that the angle $\measuredangle(s,0,s')$ is a metric, denoted by $d_{\mathbb{S}^{n-1}}(s,s')$, on the sphere $\mathbb{S}^{n-1}$, so that
$(\mathbb{S}^{n-1},d_{\mathbb{S}^{n-1}})$ is a complete metric space.

\begin{lemma}[(Uniform estimate of cones if $\Theta_{*}^{\he}(\HM^{\he},X,x)>0$)]\label{uniformestimatesoncones}
	Let $X\subset \R^{n}$, $x\in \R^{n}$ and $\Theta_{*}^{\he}(\HM^{\he},X,x)>0$. Then there is a $\rho>0$ and a mapping $s:(0,\rho)\to\mathbb{S}^{n-1}$, such that for all $\epsilon>0$ there is $c(\epsilon)>0$ with
	\begin{align*}
		c(\epsilon)r^{\he}\leq \HM^{\he}(X\cap \overline B_{r}(x)\cap C_{s(r),\epsilon}(x))\quad\text{for all }r\in (0,\rho).
	\end{align*}
\end{lemma}
\begin{proof}
	\textbf{Step 1}
		Fix $x\in \R^{n}$. Let $0<\phi<\psi$, $s\in \mathbb{S}^{n-1}$ and define
		\begin{align*}
			M(s,\alpha,\psi)\vcentcolon=\min\{\abs{I}\mid C_{s,\psi}(x)\subset \bigcup_{i\in I}C_{s_{i},\phi}(x), s_{i}\in\mathbb{S}^{n-1}, d_{\mathbb{S}^{n-1}}(s,s_{i})<\psi\}.
		\end{align*}
		As $x+\mathbb{S}^{n-1}$ is compact in $\R^{n}$ we can always find a finite subcover of $\overline C_{s,\psi}(x)$ in $\{C_{s',\phi}(x)\mid s'\in\mathbb{S}^{n-1}, d_{\mathbb{S}^{n-1}}(s,s')<\psi\}$
		and consequently $M(s,\phi,\psi)$ is finite. We can transform the situation for $s$ to that of $\tilde s$ by a rotation and hence it is clear that
		$M(s,\phi,\psi)=M(\tilde s,\phi,\psi)$ for all $s,\tilde s\in \mathbb{S}^{n-1}$. Therefore we write $M(\phi,\psi)\vcentcolon=M(s,\phi,\psi)$.\\
	\textbf{Step 2}
		We define $s_{0}(r)\vcentcolon=e_{1}=(1,0,\ldots,0)$ and $\epsilon_{0}\vcentcolon= 2\pi 2^{-0}=2\pi$.
		From Lemma \ref{implicationspositivelowerdensity} we know that there are $\rho>0$ and $c>0$, such that
		\begin{align*}
			\HM^{\he}(X\cap \overline B_{r}(X))=\HM^{\he}(X\cap \overline B_{r}(X)\cap C_{s_{0}(r),\epsilon_{0}}(x))\geq cr^{\he}\quad\text{for all }r\in (0,\rho).
		\end{align*}
		Now we set $\epsilon_{k+1}=2\pi 2^{-(k+1)}$ and find, with the help of Step 1, a direction $s_{k+1}(r)\in \mathbb{S}^{n-1}$ with $d_{\mathbb{S}^{n-1}}(s_{k}(r),s_{k+1}(r))<\epsilon_{k}$, such that
		\begin{align*}
			\MoveEqLeft
			\HM^{\he}(X\cap \overline B_{r}(X)\cap C_{s_{k+1}(r),\epsilon_{k+1}}(x))
			\geq \frac{\HM^{\he}(X\cap \overline B_{r}(X)\cap C_{s_{k}(r),\epsilon_{k}}(x))}{M(\epsilon_{k+1},\epsilon_{k})}\\
			&\geq \ldots
			\geq \frac{c}{\prod_{i=0}^{k} M(\epsilon_{i+1},\epsilon_{i})}r^{\he}\quad\text{for all }r\in (0,\rho).
		\end{align*}
		Now Lemma \ref{lemmacauchysequencescompletespace} tells us that for all $r\in (0,\rho)$ there are $s(r)\in \mathbb{S}^{n-1}$, such that $s_{k}(r)\to s(r)$, with
		\begin{align*}
			d_{\mathbb{S}^{n-1}}(s_{k}(r),s(r))\leq \sum_{i=k}^{\infty}\epsilon_{i}=\sum_{i=k}^{\infty}2\pi 2^{-i}
			=2\pi\Big[\frac{1}{1-1/2}-\frac{1-1/2^{-k}}{1-1/2}\Big]
			=2\pi 2^{-(k-1)}=\epsilon_{k-1}.
		\end{align*}
	\textbf{Step 3}
		Let $\epsilon>0$, then, as $\epsilon_{k}\to 0$, there is a $k$, such that $\epsilon>\epsilon_{k-1}+\epsilon_{k}$. Because $d_{\mathbb{S}^{n-1}}(s,s')+\phi\leq \psi$
		implies $C_{s',\phi}(x)\subset C_{s,\psi}(x)$ and we already know $d_{\mathbb{S}^{n-1}}(s_{k}(r),s(r))\leq \epsilon_{k-1}$ by Step 2, we have $C_{s_{k}(r),\epsilon_{k}}(x)\subset C_{s(r),\epsilon}(x)$
		and hence
		\begin{align*}
			\MoveEqLeft
			\HM^{\he}(X\cap \overline B_{r}(x)\cap C_{s(r),\epsilon}(x))
			\geq \HM^{\he}(X\cap \overline B_{r}(x)\cap C_{s_{k}(r),\epsilon_{k}}(x))\\
			&\geq \frac{c}{\prod_{i=0}^{k-1} M(\epsilon_{i+1},\epsilon_{i})}r^{\he}
			= c(\epsilon)r^{\he}\quad\text{for all }r\in (0,\rho).
		\end{align*}
\end{proof}

\section{Approximate tangents, counterexamples}\label{sectionapproximatetangents}

We now fix our notation regarding the tangency properties
we wish to investigate. Also we give some remarks and examples in this context. In this section we finally leave the setting of metric
spaces and are from now on only concerned with subsets of $\R^{n}$.

\begin{definition}[(Double cone in direction $s$ with opening angle $\epsilon$)]
	Let $x\in\R^{n}$, $s\in\Sphere^{n-1}$ and $\epsilon> 0$. By $C_{s,\epsilon}(x)$ we denote the open double cone centred at $x$ 
	in direction $s$, i.e.
	\begin{align*}
		C_{s,\epsilon}(x)\vcentcolon=\{y\in\R^{n}\backslash \{x\}\mid \min\{\measuredangle(y,x,x- s),\measuredangle(y,x,x+s)\}<\epsilon\}.
	\end{align*}
\end{definition}

\begin{definition}[(Weakly $\he$-linearly approximable)]
	We say that a set $X\subset \R^{n}$ is \emph{weakly $\he$-linearly approximable}, $\he\in (0,\infty)$ at a point $x\in\R^{n}$, if there is a $\rho>0$ and a mapping $s:(0,\rho)\to\mathbb{S}^{n-1}$, such that
	for every $\epsilon>0$ and every $\delta>0$, there is an  $\rho(\epsilon,\delta)\in (0,\rho)$ with
	\begin{align*}
		\HM^{\he}([X\cap \overline B_{r}(x)]\backslash C_{s(r),\epsilon}(x))\leq \delta r^{\he}\quad \text{for all }r\in (0,\rho(\epsilon,\delta)).
	\end{align*}
\end{definition}

\begin{definition}[(Weak and strong approximate $\he$-tangents)]
	Let $X\subset \R^{n}$ be a set and $x\in \R^{n}$, $\he\in(0,\infty)$. We say that $X$ has a \emph{(strong) approximate $\he$-tangent} at $x$, if there is a direction $s\in\mathbb{S}^{n-1}$, such that
	\begin{align*}
		\Theta^{\he}(\HM^{\he},X\backslash C_{s,\epsilon}(x),x)=0\quad\text{for all }\epsilon>0,
	\end{align*}
	and we say that $X$ has a \emph{weak approximate $\he$-tangent} at $x$, if there is a $\rho>0$ and a mapping $s:(0,\rho)\to\mathbb{S}^{n-1}$, such that
	\begin{align*}
		\Theta^{\he}(\HM^{\he},X\backslash C_{s(r),\epsilon}(x),x)=0\quad\text{for all }\epsilon>0.
	\end{align*}
	We will also sometimes call the direction $s$ and the mapping $s:(0,\rho)\to\mathbb{S}^{n-1}$ (strong) approximate $\he$-tangent and weak approximate $\he$-tangent, respectively.
\end{definition}

\begin{lemma}[(Weakly $\he$-linearly appr. iff weak approximate $\he$-tangents)]\label{weaklylinapproxiffweakapprtangent}
	Let $X\subset \R^{n}$ be a set and $x\in \R^{n}$, $\he\in (0,\infty)$. Then the following are equivalent
	\begin{itemize}
		\item
			$X$ is weakly $\he$-linearly approximable at $x$,
		\item
			$X$ has weak approximate $\he$-tangents at $x$.
	\end{itemize}
\end{lemma}
\begin{proof}
	One direction is directly clear from the definitions and the other direction is proven in Lemma \ref{implicationsfiniteupperdensity}.
\end{proof}

\begin{remark}[(Differences to standard use of terminology)]
	We should warn the reader, that our definition of $1$-linear approximability and approximate $1$-tangents differ from the standard use in literature 
	\cite[15.7 \& 15.10 Definition, p.206 and 15.17 Definition, p.212]{Mattila1995a} in that we refrain from imposing additional density requirements, 
	like $\Theta^{*1}(\HM^{1},X,x)>0$ in the case of approximate $1$-tangents. This is simply due to the fact that in the following sections we obtain simpler formulations of our results, because some
	distinction of cases can be omitted; as we cannot expect a set with finite curvature energy to have positive upper density at any point.
\end{remark}

\begin{remark}[(Difference between approximate $1$-tangents and tangents)]\label{differenceapptangentsandtangents}
	What it means for a set to have an approximate $1$-tangent at a point is, in some respects, quite different to having an actual tangent at this point. 
	To illustrate this, consider
	\begin{align*}
		S\vcentcolon=\{(x,0)\mid x\in [0,1]\}\cup\{(x,x^{2})\mid x\in [0,1] \}.
	\end{align*}
	As $x\mapsto x^{2}$ is convex there is $r(\epsilon)$, such that $S\cap B_{r(\epsilon)}(0)\subset C_{\epsilon}(0)$ and hence $S$ has an approximate
	$1$-tangent at $(0,0)$, but an arc length parametrisation $\gamma$ of $S$ does not posses a derivative, and hence a tangent, at $\gamma^{-1}((0,0))$.
\end{remark}

\begin{example}[(A set with weak appr. but no appr. $1$-tangents)]\label{weakbutnoapptangents}
	Set $a_{n}\vcentcolon=2^{-n^{n}n^{3}}$, $A_{n}\vcentcolon=[a_{n}/2,a_{n}]$ and
	\begin{align*}
		F\vcentcolon=\Big[\bigcup_{n\in\N}\underbrace{A_{2n}\times\{0\}}_{=\vcentcolon B_{2n}}\Big]
		\cup\Big[\bigcup_{n\in\N}\underbrace{\{0\}\times A_{2n-1}}_{=\vcentcolon B_{2n-1}}\Big].
	\end{align*}
	For $\epsilon>0$ we have
	\begin{align}\label{Xhasnoapptangentsex}
		\begin{split}
			\HM^{1}(F\cap C_{e_{1},\epsilon}(0)\cap \overline B_{a_{2n}}(0))&\geq \HM^{1}([a_{2n}/2,a_{2n}])=a_{2n}/2\\
			\HM^{1}(F\cap C_{e_{2},\epsilon}(0)\cap \overline B_{a_{2n+1}}(0))&\geq \HM^{1}([a_{2n+1}/2,a_{2n+1}])=a_{2n+1}/2.
		\end{split}
	\end{align} 
	Now (\ref{Xhasnoapptangentsex})
	tells us that no approximate $1$-tangent exists, because for every $s\in\Sphere^{n-1}$ there is $\epsilon_{s}>0$ and $i_{s}\in \{1,2\}$, such that
	$C_{e_{i_{s},\epsilon_{s}},\epsilon_{s}}(0)\cap C_{s,\epsilon_{s}}(0)=\emptyset$ and hence by (\ref{Xhasnoapptangentsex})
	there are $r_{n}=r_{n}(s)>0$, $r_{n}\to 0$ with
	\begin{align*}
		\Theta^{*1}(\HM^{1},F\backslash C_{s,\epsilon_{s}}(0),0)
		\geq \lim_{n\to\infty}\frac{\HM^{1}([F\cap C_{e_{i_{s}},\epsilon_{s}}(0)]\cap\overline B_{r_{n}}(0))}{2r_{n}}\geq \frac{1}{4}.
	\end{align*}
	On the other hand we have
	\begin{align*}
		\MoveEqLeft
		\HM^{1}([F \cap \overline B_{r}(0)]\backslash C_{e_{1},\epsilon}(0))\leq \HM^{1}([0,a_{2n+1}])=2^{-(2n+1)^{2n+1}(2n+1)^{3}}\\
		&\leq 2^{-2n}2^{-(2n)^{2n}(2n)^{3}-1}=2^{-2n}\frac{a_{2n}}{2}\leq 2^{-2n}r
	\end{align*}
	for all $r\in [a_{2n}/2,a_{2n-1}/2]$ and
	\begin{align*}
		\MoveEqLeft
		\HM^{1}([F \cap \overline B_{r}(0)]\backslash C_{e_{2},\epsilon}(0))\leq \HM^{1}([0,a_{2(n+1)}])=2^{-(2[n+1])^{2[n+1]}(2[n+1])^{3}}\\
		&\leq 2^{-(2n+1)}2^{-(2n+1)^{2n+1}(2n+1)^{3}-1}=2^{-(2n+1)}\frac{a_{2n+1}}{2}\leq 2^{-(2n+1)}r
	\end{align*}
	for all $r\in (a_{2n+1}/2,a_{2n}/2)$. We therefore have verified the definition of $F$ having a weak approximate $1$-tangent for
	\begin{align*}
		s:(0,1/2)\to\mathbb{S}^{1},\,r\mapsto
		\begin{cases}
			e_{1},&r\in \bigcup_{n\in\N}[a_{2n}/2,a_{2n-1}/2],\\
			e_{2},&r\in \bigcup_{n\in\N}(a_{2n+1}/2,a_{2n}/2).
		\end{cases}
	\end{align*}
\end{example}

One might be tempted to think that a continuum with approximate $1$-tangents is a topological $1$-manifold, i.e. a closed curve or an arc.
That these two concepts are not related can be seen by the following remark. If the reader is not familiar with the notion
of ramification order we refer him to \cite[Definition 13.5, p.442 f.]{Blumenthal1970b}.

\begin{remark}[(Relationship between appr. $1$-tangents  and ramification points)]
	If a set $M$ has an approximate $1$-tangent at $x\in M$ then $x$ can still be a ramification point.
	Let $S$ be the set from Remark \ref{differenceapptangentsandtangents}. Then $S\cup( [-1,0]\times \{0\})$ has 
	an approximate $1$-tangent at $0$ and $0$ is a point of order $3$.
	On the other hand a point of order less than $2$ does not imply that the set has an approximate $1$-tangent at this point.
	This can be sen as follows: let $M\vcentcolon=([0,1]\times\{0\})\cup (\{0\}\times [0,1])$. Then $0$ is a point of order 2 in $M$, but $M$ does not even possess 
	a weak approximate $1$-tangent at $0$.
\end{remark}

\begin{lemma}[(Density estimates for set with no approximate tangent)]\label{implicationsoflackofapptangents}
	Let $X\subset \R^{n}$, $x\in \R^{n}$, $\he\in (0,\infty)$ and $\Theta^{*\he}(\HM^{\he},X,x)>0$.
	If $X$ has no approximate $\he$-tangent at $x$,
	then there is $s\in \Sphere^{n-1}$ and $\epsilon_{0}>0$, such that
	\begin{align*}
		0<\Theta^{*\he}(\HM^{\he},X\cap C_{s,\epsilon_{0}/2}(x),x)\quad\text{and}\quad
		0<\Theta^{*\he}(\HM^{\he},X\backslash C_{s,\epsilon_{0}}(x),x).
	\end{align*}
\end{lemma}
\begin{proof}
		Assuming that there exists no approximate $\he$-tangent at $x\in X$ we know that for all directions $s\in\Sphere^{n-1}$ there is an $\epsilon_{s}>0$, 
		such that $\Theta^{*\he}(\HM^{\he},X\backslash C_{s,\epsilon_{s}}(x),x)>0$. As $\Sphere^{n-1}$ is compact 
		and $\{C_{s,\epsilon_{s}/2}(x)\}_{s\in \Sphere^{n-1}}$ is
		an open cover of $x+\Sphere^{n-1}$ there exists a finite subcover $\{C_{s_{i},\epsilon_{s_{i}}/2}(x)\}_{i=1}^{N}$. 
		Clearly this subcover also covers the 
		whole $\R^{n}\backslash\{x\}$. As $0<\Theta^{*\he}(\HM^{\he},X,x)=\Theta^{*\he}(\HM^{\he},X\backslash\{x\},x)$ we know, 
		by Lemma \ref{positiveupperdensityinfinitedecomposition}, note Remark \ref{positiveupperdensityinfinitedecompositionremark}, 
		that for some $j\in \{1,\ldots, N\}$ we have
		$\Theta^{*\he}(\HM^{\he},X\cap C_{s_{j},\epsilon_{j}/2}(x),x)>0$.
\end{proof}

\begin{lemma}[(Density estimates for set with no weak approximate tangent)]\label{densityestimaenonweaklylinapprset}
	Let $X\subset \R^{n}$, $x\in\R^{n}$, $\he\in (0,\infty)$ and $\Theta_{*}^{\he}(\HM^{\he},X,x)>0$. If $X$ has no weak approximate $\he$-tangent at $x$, 
	then there is a mapping $s:(0,\rho)\to \mathbb{S}^{n-1}$, $\rho>0$ and $\epsilon_{0}>0$, such that
	\begin{align*}
		0<\Theta_{*}^{\he}(\HM^{\he},X\cap C_{s(r),\epsilon_{0}/2}(x),x)\quad\text{and}\quad
		0<\Theta^{*\he}(\HM^{\he},X\backslash C_{s(r),\epsilon_{0}}(x),x).
	\end{align*}
\end{lemma}
\begin{proof}
	If $X$ has no weak approximate $\he$-tangent at $x\in\R^{n}$ it is not weakly $\he$-linearly approximable in $x$, by Lemma \ref{weaklylinapproxiffweakapprtangent}, 
	so that for all $\rho>0$ and all mappings $s:(0,\rho)\to \mathbb{S}^{n-1}$ there is an $\epsilon_{0}>0$ and a $\delta_{0}>0$, 
	such that for all $\rho'\in (0,\rho)$ there is $r\in (0,\rho')$ with
	\begin{align*}
		\delta_{0} r^{\he}<\HM^{\he}([X\cap \overline B_{r}(x)]\backslash C_{s(r),\epsilon_{0}}(x)).
	\end{align*}
	By choosing $\rho'=\rho (2k)^{-1}$ we obtain a sequence $(r_{k})_{k\in\N}$, $r_{k}>0$, $r_{k}\to 0$, with
	\begin{align}\label{conditionwhichimpliesposupperdensity}
		\delta_{0} r_{k}^{\he}<\HM^{\he}([X\cap \overline B_{r_{k}}(x)]\backslash C_{s(r_{k}),\epsilon_{0}}(x)) \quad\text{for all }k\in\N.
	\end{align}
	Now fix $\rho$ and $s:(0,\rho)\to\mathbb{S}^{n-1}$ to be those we obtain from Lemma \ref{uniformestimatesoncones}. Then
	\begin{align*}
		0<\Theta^{*\he}(\HM^{\he},X\backslash C_{s(r),\epsilon_{0}}(x),x).
	\end{align*}
	by (\ref{conditionwhichimpliesposupperdensity}) and Lemma \ref{uniformestimatesoncones} gives us
	\begin{align*}
		0<c(\epsilon_{0}/2)/2\leq \Theta_{*}^{\he}(\HM^{\he},X\cap C_{s(r),\epsilon_{0}/2}(x),x).
	\end{align*}
\end{proof}

We shall now give a construction that guarantees that a set has no weak approximate $\he$-tangent.

\begin{lemma}[(Construction of sets with no weak appr. tangent)]\label{constructionnoweakapprtangent}
	Let $X\subset \R^{n}$, $\he\in (0,\infty)$ such that $X$ has an approximate $\he$-tangent in direction $s\in \mathbb{S}^{n-1}$ at $x\in \R^{n}$ and $\Theta^{*\he}(\HM^{\he},X,x)>0$.
	Let $A\in \mathrm{SO}(n)\backslash\{\mathrm{id}\}$, $\phi_{A}(y)=A(y-x)+x$ such that the axis of rotation does not coincide with $\R s$.
	Then $X\cup \phi_{A}(X)$ has no weak approximate $\he$-tangent at $x$.
\end{lemma}
\begin{proof}
	Clearly $\phi_{A}(X)$ has an approximate $\he$-tangent in direction $\phi_{A}(s)$, so that $X\cup \phi_{A}(X)$ cannot have a weak
	approximate $\he$-tangent.
\end{proof}

\section{Finite \texorpdfstring{$1/\Delta$}{inverse thickness}, \texorpdfstring{$\U_{p}^{\he}$}{Up alpha} implies app. \texorpdfstring{$\he$}{alpha}-tangents for \texorpdfstring{$p\in [\he,\infty)$}{p in [alpha,infinity)}}\label{sectionUp}

We now show that for $p\in[\he,\infty)$ a set with finite $\mathcal{U}_{p}^{\he}$ is guaranteed to have approximate $\he$-tangents at all points.
This directly implies similar results for the inverse thickness $1/\Delta$. Later on we give a 
counterexample to the analogous result for $\he=1$ and $p\in (0,1)$.

\begin{lemma}[(Finite $\U_{p}^{\he}$ guarantees approximate $\he$-tangents)]\label{finiteUpguaranteesapproxiamtetangents}
	Let $X\subset \R^{n}$, $x\in \R^{n}$, $\he\in (0,\infty)$, $p\in [\he,\infty)$ and $\U_{p}^{\he}(X)<\infty$.
	Then $X$ has an approximate $\he$-tangent at $x$.
\end{lemma}
\begin{proof}
		Assume that $\Theta^{*\he}(\HM^{\he},X,x)>0$ -- which we might without loss of generality, because else the proposition is clear -- and that $X$ has no approximate $\he$-tangent at $x$.
		As $x$ has to be an accumulation point of $X$ we can, by means of Lemma \ref{curvatureenergiesifaccpointisremovedlemma}, assume that 
		without loss of generality $x\in X$. By Lemma \ref{finiteenergyfinitemeasureball} and Lemma \ref{consequencesfiniteenergyhe<1} we can also assume that $\HM^{\he}(X\cap B_{r}(x))<\infty$ for all small radii.
		Now we use Lemma \ref{implicationsoflackofapptangents} and set
		$A\vcentcolon=X\cap C_{s,\epsilon/2}(x) $, $B\vcentcolon= X\backslash C_{s,\epsilon}(x)$ 
		and choose a sequence of radii $r_{n}\downarrow 0$, such that $\HM^{\he}(A\cap\overline B_{r_{n}}(x))/r_{n}^{\he}\geq c>0$.
		Then $\measuredangle(a,x,b)\in [\epsilon/2,\pi-\epsilon/2]$ for all $a\in A$ and all $b\in B$. 
		Clearly $x$ is an accumulation point of of $B$, so that
		for each $n\in\N$ there exists $b_{n}\in B \cap \overline B_{r_{n}}(x)$. Using Lemma \ref{distanceintermsofangle} we obtain for all 
		$a\in A \cap \overline B_{r_{n}}(x)\backslash\{x\}$
		\begin{align*}
			\MoveEqLeft \kappa_{G}(a)\geq\frac{1}{r(a,b_{n},x)}=\frac{2\dist(L_{a,b_{n}},x)}{\norm{a-x}\norm{b_{n}-x}}
			\geq\frac{\sin(\epsilon/2)\min\{\norm{a-x},\norm{b_{n}-x}\}}{\norm{a-x}\norm{b_{n}-x}}\\
			&=\frac{\sin(\epsilon/2)}{\max\{\norm{a-x},\norm{b_{n}-x}\}}\geq \frac{\sin(\epsilon/2)}{r_{n}}.
		\end{align*}
		We have
		\begin{align*}
			\MoveEqLeft \U_{p}^{\he}(B_{2r_{n}}(x)\cap X)\geq
			\int_{A\cap\overline B_{r_{n}}(x)\backslash\{x\}} \kappa_{G}^{p}(t)\dHM(t)
			\geq \HM^{\he}(A\cap\overline B_{r_{n}}(x)) \Big(\frac{\sin(\epsilon/2)}{r_{n}}\Big)^{p}\\
			&\geq c r_{n}^{\he}\Big(\frac{\sin(\epsilon/2)}{r_{n}}\Big)^{p}\geq c'>0
		\end{align*}
		for all $n\in\N$. Hence Lemma \ref{nicebehaviourlimitofsmallballs} tells us that $\mathcal{U}_{p}^{\he}(X)=\infty$,
		note that for this we needed $\HM^{\he}(B_{2r_{n}}(x)\cap X)<\infty$. This is absurd as $\mathcal{U}_{p}^{\he}(X)<\infty$.
\end{proof}

\begin{corollary}[(Sets with finite $\U_{p}^{1}$ are rectifiable)]\label{finiteUp=rectifiable}
	Let $X\subset \R^{n}$ be an $\HM^{1}$-measurable set and $p\in [1,\infty)$. 
	If $\U_{p}^{1}(X)<\infty$ then $X$ is $1$-rectifiable.
\end{corollary}
\begin{proof}
	For all $n\in\N$ Lemma \ref{finiteenergyfinitemeasureball} tells us that $X\cap B_{n}(0)$ has finite measure, so that by 
	Lemma \ref{finiteUpguaranteesapproxiamtetangents} together with the equivalent characterisation of rectifiablity in terms of 
	approximate $1$-tangents, see for example \cite[15.19 Theorem, p. 212]{Mattila1995a}, we know that all $X\cap B_{n}(0)$ are rectifiable.
	By taking all the rectifiable curves that cover the $X\cap B_{n}(0)$, which are still countably many, we have covered $X$ with
	countably many curves, so that $X$ is rectifiable.
\end{proof}

\begin{corollary}[(Sets with positive thickness are rectifiable)]
	Let $X\subset \R^{n}$ be an $\HM^{1}$-measurable set with $\HM^{1}(X)<\infty$ and $1/\Delta[X]<\infty$. Then $X$ is $1$-rectifiable and has 
	an approximate $1$-tangent at each point $x\in \R^{n}$.
\end{corollary}
\begin{proof}
	Because $\U_{p}^{1}(X)\leq [\HM^{1}(X)]^{p}/\Delta[X]$, see Lemma \ref{inequalityforcurvatureenergies}, this a an immediate consequence of 
	Lemma \ref{finiteUp=rectifiable}. The result for the approximate $1$-tangents remains true when $X$ is not measurable, but meets the other
	hypotheses.
\end{proof}

\subsection{Finite \texorpdfstring{$\mathcal{U}_{p}^{1}$}{Up1} does not imply (weak) approx. tangents for \texorpdfstring{$p\in (0,1)$}{p in (0,1)}}

For further reference we define

\begin{definition}[(The set $E$)]\label{definitionTset}
	We set $E\vcentcolon=([-1,1]\times\{0\})\cup(\{0\}\times [0,1])\subset \R^{2}$ as well as $E_{1}\vcentcolon=[-1,0]\times\{0\}$, 
	$E_{2}\vcentcolon=\{0\}\times[0,1]$ and $E_{3}\vcentcolon=[0,1]\times\{0\}$.
\end{definition}

Clearly $E$ does not have a weak approximate $1$-tangent at $(0,0)$.
To show that our results are sharp, we need to compute the appropriate energy of $E$ in each section. We therefore start with

\begin{proposition}[(The set $E$ has finite $\mathcal{U}_{p}^{1}$ for $p\in (0,1)$)]\label{TsetfiniteUp}
	For $p\in (0,1)$ we have
	\begin{align*}
		\U_{p}^{1}(E)\leq \frac{6}{1-p}.
	\end{align*}
\end{proposition}
\begin{proof}
	For all $x\in E\backslash\{0\}$ and $y,z\in B_{\norm{x}}(x)\cap E$, $y\not=z$ 
	we have $\kappa(x,y,z)=0$, so that for $\kappa(x,y,z)>0$ we need $\norm{x-y}\geq\norm{x}$ or
	$\norm{x-z}\geq \norm{x}$, which both result in $r(x,y,z)\geq\norm{x}/2$ and consequently 
	\begin{align*}
		\sup_{\substack{y,z\in E\backslash\{x\}\\ y\not= z }} \kappa(x,y,z) \leq \frac{2}{\norm{x}},
	\end{align*}
	so that for $p\in (0,1)$
	\begin{align*}
		\MoveEqLeft
		\mathcal{U}_{p}^{1}(E)
		=\int_{E\backslash\{0\}} \big(\underbrace{\sup_{\substack{y,z\in E\backslash\{x\}\\ y\not= z }} 
		\kappa(x,y,z)}_{\leq 2/\norm{x}}\big)^{p}\dd\HM^{1}(x)\\
		&\leq 3\int_{E_{2}}\frac{2}{\norm{x}^{p}}\dd\HM^{1}(x)
		= 6\int_{0}^{1}\frac{1}{s^{p}}\dd\mathcal{L}^{1}(s)
		=\frac{6}{1-p}<\infty.
	\end{align*}
\end{proof}

\section{Finite \texorpdfstring{$\mathcal{I}_{p}^{\he}$}{Ip alpha} implies weak app. tangents for \texorpdfstring{$p\in [2\he,\infty)$}{p in [2 alpha,infinity)}}\label{sectionIp}

The purpose of this section is to show that for $p\in [2\he,\infty)$ a set with finite $\mathcal{I}_{p}^{\he}$ has a weak approximate $\he$-tangent
at all points where the lower density is positive. We also show that this is not true if $\he=1$ and $p\in (0,2)$.

\begin{lemma}[(Necessary conditions for finite $\mathcal{I}_{p}^{\he}$)]\label{setseparatedbyanglehasinfiniteIpenergy}
	Let $X\subset \R^{n}$, $z_{0}\in \R^{n}$, $\he\in (0,\infty)$, $\HM^{\he}(X)<\infty$. Let $\epsilon>0$, $c>0$ and two sequences of sets $A_{n},B_{n}\subset X$ as well as a sequence 
	$(r_{n})_{n\in\N}$, $r_{n}>0$, $r_{n}\to 0$ be given, with the following properties:
	\begin{itemize}
		\item
			for all $n\in\N$ and all $x\in A_{n}\backslash \{z_{0}\}$ and $y\in B_{n}\backslash\{z_{0}\}$ we have $\pi-\epsilon\geq\measuredangle(x,z_{0},y)\geq \epsilon$,
		\item
			for all $n\in\N$ we have
			\begin{align*}
				cr_{n}^{\he}\leq \min\{\HM^{\he}(A_{n}\cap \overline B_{r_{n}}(z_{0})),\HM^{\he}(B_{n}\cap \overline B_{r_{n}}(z_{0}))\}.
		\end{align*}	
	\end{itemize}
	Then $\mathcal{I}_{p}^{\he}(X)=\infty$ for all $p\geq 2\he$.
\end{lemma}
\begin{proof}
	Let $p\geq 2\he$ and suppose for contradiction that $\mathcal{I}_{p}^{\he}(X)<\infty$.
	As $z_{0}$ has to be an accumulation point of $X$ we can, by means of Lemma \ref{curvatureenergiesifaccpointisremovedlemma}, assume that 
	without loss of generality $z_{0}\in X$.
	If we set
	\begin{align*}
		\tilde A_{n}\vcentcolon=A_{n}\cap \overline B_{r_{n}}(z_{0})\quad\text{and}\quad
		\tilde B_{n}\vcentcolon=B_{n}\cap \overline B_{r_{n}}(z_{0})
	\end{align*} 
	Lemma \ref{distanceintermsofangle} gives us
	\begin{align*}
		\MoveEqLeft \kappa_{i}(x,y)\geq\kappa(x,y,z_{0})=\frac{2\dist(L_{x,y},z_{0})}{\norm{x-z_{0}}\norm{y-z_{0}}}
		\geq\frac{\sin(\epsilon/2)\min\{\norm{x-z_{0}},\norm{y-z_{0}}\}}{\norm{x-z_{0}}\norm{y-z_{0}}}\\
		&=\frac{\sin(\epsilon/2)}{\max\{\norm{x-z_{0}},\norm{y-z_{0}}\}}\geq \frac{\sin(\epsilon/2)}{r_{n}},
	\end{align*}
	for all $x\in \tilde A_{n}\backslash\{z_{0}\}$ and $y\in \tilde B_{n}\backslash\{z_{0}\}$. Now we have
	\begin{align*}
		\MoveEqLeft \mathcal{I}_{p}^{\he}(X\cap B_{2r_{n}}(z_{0}))\geq \mathcal{I}_{p}^{\he}(X\cap \overline B_{r_{n}}(z_{0}))
		=\int_{X\cap \overline B_{r_{n}}(z_{0})}\int_{X\cap \overline B_{r_{n}}(z_{0})}\kappa_{i}^{p}(x,y)\dd\HM^{\he}(x)\dd\HM^{\he}(y)\\
		&\geq \int_{\tilde B_{n}}\int_{\tilde A_{n}}\kappa_{i}^{p}(x,y)\dd\HM^{\he}(x)\dd\HM^{\he}(y)
		\geq \HM^{\he}(\tilde B_{n})\HM^{\he}(\tilde A_{n})\Big(\frac{\sin(\epsilon/2)}{r_{n}}\Big)^{p}\\
		&\geq c^{2}\sin^{p}(\epsilon/2) r_{n}^{2\he-p}\geq c'>0
	\end{align*}
	for $p\geq 2\he$ and all $n\in\N$. Hence Lemma \ref{nicebehaviourlimitofsmallballs} tells us that $\mathcal{I}_{p}^{\he}(X)=\infty$,
	note that for this we needed $\HM^{\he}(B_{2r_{n}}(x)\cap X)<\infty$. This is absurd as we assumed $\mathcal{I}_{p}^{\he}(X)<\infty$.
\end{proof}

\begin{proposition}[(Finite $\mathcal{I}_{p}^{\he}$, $p\geq 2\he$ implies weak app. $\he$-tangents)]
	Let $X\subset \R^{n}$ be a set, $\he\in (0,\infty)$ and $x\in \R^{n}$ with $0<\Theta_{*}^{\he}(\HM^{\he},X,x)$. If $p\in [2\he,\infty)$ and $\mathcal{I}_{p}^{\he}(X)<\infty$ 
	then $X$ has a weak approximate $\he$-tangent at $x$.
\end{proposition}
\begin{proof}
	Assume that this is not the case. 
	By Lemma \ref{finiteenergyfinitemeasureball} and Lemma \ref{consequencesfiniteenergyhe<1} we can without loss of generality assume that $\HM^{\he}(X\cap \overline B_{r}(x))<\infty$ for all small radii.
	Then by Lemma \ref{densityestimaenonweaklylinapprset} there is a mapping $s:(0,\rho)\to \mathbb{S}^{n-1}$, $\rho>0$ and $\epsilon_{0}>0$, such that
	\begin{align*}
		0<\Theta_{*}^{\he}(\HM^{\he},X\cap C_{s(r),\epsilon_{0}/2}(x),x)
	\end{align*}
	and
	\begin{align*}
		0<\Theta^{*\he}(\HM^{1},X\backslash C_{s(r),\epsilon_{0}}(x),x).
	\end{align*}
	This means that there is a constant $c>0$ and a sequence $(r_{n})_{n\in\N}$, $r_{n}>0$, $r_{n}\to 0$, such that
	\begin{align*}
		cr_{n}^{\he}\leq \min\{\HM^{\he}([X\cap C_{s(r_{n}),\epsilon_{0}/2}(x)]\cap \overline B_{r_{n}}(x)),\HM^{\he}([X\backslash C_{s(r_{n}),\epsilon_{0}}(x)]\cap \overline B_{r_{n}}(x))\}
	\end{align*}
	and hence the hypotheses of Lemma \ref{setseparatedbyanglehasinfiniteIpenergy} hold for
	\begin{align*}
		A_{n}\vcentcolon=[X\cap \overline B_{r}(x)]\cap C_{s(r_{n}),\epsilon_{0}/2}(x)\quad\text{and}\quad
		B_{n}\vcentcolon=[X\cap \overline B_{r}(x)]\backslash C_{s(r_{n}),\epsilon_{0}}(x)
	\end{align*}
	for $r$ small enough, i.e. the role $X$ in Lemma \ref{setseparatedbyanglehasinfiniteIpenergy} is played by $X\cap \overline B_{r}(x)$, and $\epsilon\vcentcolon=\epsilon_{0}/2$, so that we have proven the proposition.
\end{proof}

\subsection{Finite \texorpdfstring{$\mathcal{I}_{p}^{1}$}{Ip1} does not imply (weak) approx. tangents for \texorpdfstring{$p\in (0,2)$}{p in (0,2)}}

\begin{proposition}[(The set $E$ has finite $\mathcal{I}_{p}^{1}$ for $p\in (1,2)$)]\label{TsetfiniteIp}
	Let $E$ be the set from Definition \ref{definitionTset}. For $p\in (1,2)$ we have
	\begin{align*}
		\mathcal{I}_{p}^{1}(E)\leq \frac{9\cdot 2^{3p/2+1}(2^{1-p}-1)}{(1-p)(2-p)}.
	\end{align*}
\end{proposition}
\begin{proof}
	Let $x,y\in E\backslash \{0\}$, $x\not= y$. We are interested in the maximal value of $\kappa(x,y,z)$ for $z\in E\backslash\{x,y\}$. 
	As $\kappa$ is invariant under isometries we can restrict ourselves to the cases 
	$x,y\in E_{1}$ and $x\in E_{1}$, $y\in E_{3}$ and $x\in E_{1}$, $y\in E_{2}$.	 
	In each of these cases we want to estimate $\kappa(x,y,z)$ independently of $z$.
	We denote the non-zero components of $x,y,z$ by $\xi,\eta, \zeta$ respectively.\\
	\textbf{Case 1}
		If $x,y\in E_{1}$, $xy\not=0$ we clearly can assume $z\in E_{2}\backslash\{0\}$ and hence
		\begin{align*}
			\kappa(x,y,z)=\frac{2\zeta}{\sqrt{\xi^{2}+\zeta^{2}}\sqrt{\eta^{2}+\zeta^{2}}}
			=\frac{2}{\sqrt{\zeta^{2}+\xi^{2}+\eta^{2}+\xi^{2}\eta^{2}/\zeta^{2} }}.
		\end{align*}
		By taking first and second derivatives of $f(u)=\alpha u+\beta/u$, $\alpha,\beta>0$, 
		we easily see that $\min_{u>0}f(u)=f(\sqrt{\beta/\alpha})$, so that for all $\zeta >0$ we have
		\begin{align*}
			\zeta^{2}+\frac{\xi^{2}\eta^{2}}{\zeta^{2}}\geq \xi \eta+\frac{\xi^{2}\eta^{2}}{\xi \eta}=2\xi \eta
		\end{align*}
		and therefore
		\begin{align*}
			\kappa(x,y,z)\leq \frac{2}{\sqrt{\xi^{2}+\eta^{2}+2\xi \eta}}=\frac{2}{\abs{\xi}+\abs{\eta}}.
		\end{align*}
	\textbf{Case 2}
		If $x\in E_{1}$, $y\in E_{3}$, $xy\not=0$ we do need $z\in E_{2}$ in order to have $\kappa(x,y,z)>0$, but then
		$\kappa(x,y,z)=\kappa(x,-y,z)$, so that we can without loss of generality assume that $y\in E_{1}$.
		This was already done in Case 1.\\
	\textbf{Case 3}
		If $x\in E_{1}$, $y\in E_{2}$, $xy\not=0$ we note that we have $\kappa(x,y,z)=\kappa(x,y,-z)$ for $z\in E_{3}$,
		so that we may assume $z\in E_{1}$ without loss of generality. Then
		\begin{align*}
			\kappa(x,y,z)=\frac{2\eta}{\sqrt{\xi^{2}+\eta^{2}}\sqrt{\zeta^{2}+\eta^{2}}}\leq 
			\frac{2\eta}{\sqrt{\xi^{2}+\eta^{2}}\sqrt{\eta^{2}}}=\frac{2}{\sqrt{\xi^{2}+\eta^{2}}}
			\leq \frac{2\sqrt{2}}{\abs{\xi}+\eta}.
		\end{align*}

	In all cases we have
	\begin{align*}
		\kappa(x,y,z)\leq \frac{2\sqrt{2}}{\abs{\xi}+\abs{\eta}}\quad\text{for all }z\in E\backslash\{x,y\},
	\end{align*}
	which for $p\in (1,2)$ gives us
	\begin{align*}
		\MoveEqLeft\mathcal{I}_{1}^{2}(E)
		\leq
		9\cdot 2^{3p/2}\int_{0}^{1}\int_{0}^{1}\Big(\frac{1}{s+t}\Big)^{p}\dd\mathcal{L}^{1}(s)\dd\mathcal{L}^{1}(t)\\
		&=\frac{9\cdot 2^{3p/2}}{1-p}\int_{0}^{1} [(1+t)^{1-p}-t^{1-p}]\dd\mathcal{L}^{1}(t)
		=\frac{9\cdot 2^{3p/2}}{(1-p)(2-p)} \Big[[(1+t)^{2-p}-t^{2-p}]\Big]_{0}^{1}\\
		&=\frac{9\cdot 2^{3p/2}}{(1-p)(2-p)} \Big[[2^{2-p}-1]-[1-0]\Big]
		=\frac{9\cdot 2^{(3p/2)+1}(2^{1-p}-1)}{(1-p)(2-p)}.
	\end{align*}
\end{proof}

\begin{corollary}[(The set $E$ has finite $\mathcal{I}_{p}^{1}$ for $p\in (0,2)$)]\label{TsetfiniteIp<2}
	For $p\in (0,2)$ we have $\mathcal{I}_{p}^{1}(E)<\infty$.
\end{corollary}
\begin{proof}
	This is a consequence of Lemma \ref{TsetfiniteIp} together with $\HM^{1}(E)=3$ and Lemma \ref{comparisonofcurvatureenergiesfordifferentp}.
\end{proof}

\section{Finite \texorpdfstring{$\mathcal{M}_{p}^{\he}$}{Mp alpha} implies weak app. tangents for \texorpdfstring{$p\in [3\he,\infty)$}{p in [3 alpha,infinity)}}\label{sectionMp}

In this section we show that for $p\in [3\he,\infty)$ a set with finite upper density and finite $\mathcal{M}_{p}^{\he}$ has a weak approximate $\he$-tangent
at all points where the lower density is positive. After this we demonstrate that this is not true for $\he=1$ and $p\in (0,3)$.

\begin{lemma}[(Necessary conditions for finite Menger curvature)]\label{setseparatedbyanglehasinfinitemengercurvature}
	Let $X\subset \R^{n}$, $z_{0}\in \R^{n}$, $\he\in (0,\infty)$, $\HM^{\he}(X)<\infty$, $\Theta_{*}^{\he}(\HM^{\he},X,z_{0})>0$. Let $\epsilon>0$, $c>0$, $q_{0}\in (0,1)$ and 
	two sequences of sets $A_{n},B_{n}\subset X$ as well as a sequence 
	$(r_{n})_{n\in\N}$, $r_{n}>0$, $r_{n}\to 0$ be given, with the following properties:
	\begin{itemize}
		\item
			for all $n\in\N$ and all $x\in A_{n}\backslash \{z_{0}\}$ and $y\in B_{n}\backslash\{z_{0}\}$ we have $\pi-\epsilon\geq\measuredangle(x,z_{0},y)\geq \epsilon$,
		\item
			for all $n\in\N$ we have
			\begin{align*}
				c r_{n}^{\he}\leq \min\{\HM^{\he}(A_{n}\cap [\overline B_{r_{n}}(z_{0})\backslash B_{q_{0}r_{n}}(z_{0})]),
				\HM^{\he}(B_{n}\cap [\overline B_{r_{n}}(z_{0})\backslash B_{q_{0}r_{n}}(z_{0})])\}.
		\end{align*}	
	\end{itemize}
	Then $\M_{p}^{\he}(X)=\infty$ for all $p\geq 3\he$.
\end{lemma}
\begin{proof}
	Let $p\geq 3\he$ and suppose for contradiction that $\mathcal{M}_{p}^{\he}(X)<\infty$. 
	We set 
	\begin{align*}
		\tilde A_{n}\vcentcolon=A_{n}\cap [\overline B_{r_{n}}(z_{0})\backslash B_{q_{0}r_{n}}(z_{0})]\quad\text{and}\quad
		\tilde B_{n}\vcentcolon=B_{n}\cap [\overline B_{r_{n}}(z_{0})\backslash B_{q_{0}r_{n}}(z_{0})].
	\end{align*}
	Considering Lemma \ref{distanceintermsofangle} we know that for all $x\in \tilde A_{n}\backslash\{z_{0}\}$ 
	and $y\in \tilde B_{n}\backslash\{z_{0}\}$ we have
	$\dist(L_{x,y},z_{0})\geq \sin(\epsilon) q_{0}r_{n}/2$ and therefore for all $z\in B_{\sin(\epsilon)q_{0}r_{n}/4}(z_{0})$
	\begin{align*}
		\dist(L_{x,y},z)\geq \dist(L_{x,y},z_{0})-d(z_{0},z)\geq \frac{\sin(\epsilon)}{4}q_{0}r_{n}.
	\end{align*}
	There exists a constant $c_{1}>0$, such that
	\begin{align*}
		c_{1}(\sin(\epsilon)q_{0}r_{n}/4)^{\he}\leq\HM^{\he}(X\cap \overline B_{\sin(\epsilon)q_{0}r_{n}/4}(z_{0}))
	\end{align*}
	for all $n\in\N$. Then
	\begin{align*}
		\MoveEqLeft[1] 
		\M_{p}^{\he}(X\cap B_{2r_{n}}(z_{0}))\\
		&\geq \int_{X\cap \overline B_{\sin(\epsilon)q_{0}r_{n}/4}(z_{0})}\int_{\tilde A_{n}}\int_{\tilde B_{n}}
		\Big(\frac{2\dist(L_{x,y},z)}{\norm{x-z}\norm{y-z}}\Big)^{p}\dHM(x)\dHM(y)\dHM(z)\\
		&\geq \int_{X\cap \overline B_{\sin(\epsilon)q_{0}r_{n}/4}(z_{0})}\int_{\tilde A_{n}}\int_{\tilde B_{n}}
		\Big(\frac{2\frac{\sin(\epsilon)}{4}q_{0}r_{n} }{4r_{n}^{2}}\Big)^{p}\dHM(x)\dHM(y)\dHM(z)\\
		&\geq  \Big(\frac{\sin(\epsilon)q_{0}}{8}\Big)^{p} \HM^{\he}(X\cap \overline B_{\sin(\epsilon)q_{0}r_{n}/4}(z_{0}))
		\HM^{\he}(\tilde A_{n}) \HM^{\he}(\tilde B_{n})
		\Big(\frac{1}{r_{n}}\Big)^{p}\\
		&\geq  \Big(\frac{\sin(\epsilon)q_{0}}{8}\Big)^{p} c_{1}\Big(\frac{\sin(\epsilon)q_{0}r_{n}}{4}\Big)^{\he}
		c^{2}r_{n}^{2\he}\Big(\frac{1}{r_{n}}\Big)^{p}
		\geq  \Big(\frac{\sin(\epsilon)q_{0}}{8}\Big)^{p+\he}2^{\he}c_{1}c^{2}r_{n}^{3\he-p}\geq c'>0
	\end{align*}
	for all $n\in\N$. Hence Lemma \ref{nicebehaviourlimitofsmallballs} tells us that $\mathcal{M}_{p}^{\he}(X)=\infty$,
	note that for this we needed $\HM^{\he}(B_{2r_{n}}(x)\cap X)<\infty$. This is absurd as we assumed $\mathcal{M}_{p}^{\he}(X)<\infty$.
\end{proof}

\begin{proposition}[(Finite $\M_{p}^{\he}$, $p\geq 3\he$ implies weak appr. tangents if $\Theta^{*\he}$ is finite)]
	Let $X\subset \R^{n}$ be a set, $\he\in (0,\infty)$ and $x\in \R^{n}$ with $0<\Theta_{*}^{\he}(\HM^{\he},X,x)\leq \Theta^{*\he}(\HM^{\he},X,x)<\infty$. 
	If $p\in [3\he,\infty)$ and $\M_{p}^{\he}(X)<\infty$ then $X$ has a weak approximate $\he$-tangent at $x$.
\end{proposition}
\begin{proof}
	Assume that this is not the case. 
	By Lemma \ref{finiteenergyfinitemeasureball} and Lemma \ref{consequencesfiniteenergyhe<1} we can without loss of generality assume that $\HM^{\he}(X\cap \overline B_{r}(x))<\infty$ for all small radii.
	Then by Lemma \ref{densityestimaenonweaklylinapprset} there is a mapping $s:(0,\rho)\to \mathbb{S}^{n-1}$, $\rho>0$ and $\epsilon_{0}>0$, such that
	\begin{align*}
		0<\Theta_{*}^{\he}(\HM^{\he},X\cap C_{s(r),\epsilon_{0}/2}(x),x)
	\end{align*}
	and
	\begin{align*}
		0<\Theta^{*\he}(\HM^{\he},X\backslash C_{s(r),\epsilon_{0}}(x),x).
	\end{align*}
	This means that the hypotheses of Lemma \ref{simultaneousestimateannuli} hold for
	\begin{align*}
		A(r)\vcentcolon=X\cap C_{s(r),\epsilon_{0}/2}(x)\quad\text{and}\quad 
		B(r)\vcentcolon=X\backslash C_{s(r),\epsilon_{0}}(x),
	\end{align*}
	so that there exists a $q_{0}\in (0,1)$, a sequence $(r_{n})_{n\in\N}$, $r_{n}>0$, $\lim_{n\to\infty}r_{n}=0$ and a constant $c>0$ such that 
	\begin{align*}
		c r_{n}^{\he}\leq \min\{\HM^{\he}(A(r_{n})\cap[\overline B_{r_{n}}(x)\backslash B_{q_{0}r_{n}}(x)]),
		\HM^{\he}(B(r_{n})\cap [\overline B_{r_{n}}(x)\backslash B_{q_{0}r_{n}}(x)])\}.
	\end{align*}
	Hence the hypotheses of Lemma \ref{setseparatedbyanglehasinfinitemengercurvature} are fulfilled for $\epsilon\vcentcolon=\epsilon_{0}/2$, note that $\HM^{\he}(X\cap \overline B_{r}(x))<\infty$ for small radii,
	and we have proven the proposition.
\end{proof}

\subsection{Finite \texorpdfstring{$\mathcal{M}_{p}^{1}$}{Mp1} does not imply (weak) app. tangents for \texorpdfstring{$p\in (0,3)$}{p in (0,3)}}

\begin{definition}[(The functional $\F_{p}$)]
	For $A,B,C\subset\R^{n}$ measurable, $p>0$ we set
	\begin{align*}
		\F_{p}(A,B,C)\vcentcolon= \int_{C}\int_{B}\int_{A}\kappa^{p}(x,y,z)\dd\HM^{1}(x)\dd\HM^{1}(y)\dd\HM^{1}(z).
	\end{align*}
\end{definition}

\begin{remark}[($\F_{p}$ is invariant under permutations)]\label{permutation}
	By Fubini's Theorem and the symmetry of the integrand under permutations, as well as its measurability it is clear that for all measurable subsets 
	$A,B,C\subset X$ of $X\subset \R^{n}$ we have
	\begin{align*}
		\F_{p}(A,B,C)=\F_{p}(B,C,A)=\F_{p}(C,A,B)=\F_{p}(B,A,C)=\F_{p}(A,C,B)=\F_{p}(C,B,A).
	\end{align*}
\end{remark}

\begin{proposition}[(The set $E$ has finite $\M_{p}^{1}$ for $p\in [2,3)$)]\label{finiteMpbutramificationpoints}
	Let $E$ be the set from Definition \ref{definitionTset}. For $p\in [2,3)$ we have
	\begin{align*}
		\M_{p}^{1}(E)\leq \frac{72\pi}{(3-p)^{2}}.
	\end{align*}
\end{proposition}
\begin{proof}
	\textbf{Step 1} 
		By Lemma \ref{decompositionoftripleintegral} it is clear that
		\begin{align*}
			\M_{p}^{1}(E)=\sum_{\mathclap{i,j,k\in\{1,2,3\}}}\F_{p}(E_{i},E_{j},E_{k}).
		\end{align*}
		Since the integrand $\kappa^{p}$ vanishes on certain sets, we have
		\begin{align*}
			\sum_{\mathclap{\substack{i,j,k\in\{1,2,3\}\\\#\{i,j,k\}=1}}}\, \F_{p}(E_{i}, E_{j}, E_{k}) 
			+\sum_{\mathclap{i,j,k\in\{1,3\}}}\, \F_{p}(E_{i}, E_{j}, E_{k})=0,
		\end{align*}
		furthermore
		\begin{align*}
			\M_{p}^{1}(E_{1}\cup E_{2})=\sum_{\mathclap{\substack{i,j,k\in\{1,2\}\\\#\{i,j,k\}=2}}}\, \F_{p}(E_{i}, E_{j}, E_{k})=
			\sum_{\mathclap{\substack{i,j,k\in\{2,3\}\\\#\{i,j,k\}=2}}}\, \F_{p}(E_{i}, E_{j}, E_{k})=\M_{p}^{1}(E_{2}\cup E_{3}),
		\end{align*}
		as the energy is invariant under isometries. Considering Remark \ref{permutation} we obtain
		\begin{align*}
			\M_{p}^{1}(E_{1}\cup E_{2})=\M_{p}^{1}(E_{2}\cup E_{3})=3(\F_{p}(E_{1},E_{1},E_{2})+\F_{p}(E_{1},E_{2},E_{2}))=6\F_{p}(E_{1},E_{1},E_{2}),
		\end{align*}
		where the last equality is, again, due to the invariance of the integrand under isometries. By considering the integrand $\kappa^{p}$
		in the form
		\begin{align*}
			\kappa^{p}(x,y,z)=\bigg( \frac{2\dist(x,L_{zy})}{d(x,y)d(x,z)} \bigg)^{p}
		\end{align*}
		for $x\in E_{2}, y\in E_{1}$ and $z\in E_{3}$ we note, that $\kappa^{p}(x,y,z)=\kappa^{p}(x,y,-z)$, 
		by mapping $E_{3}$ onto $E_{1}$ via $z\mapsto -z$ we find
		\begin{align*}
			\F_{p}(E_{3},E_{1},E_{2})=\F_{p}(E_{1},E_{1},E_{2}),
		\end{align*}
		so that
		\begin{align*}
			\sum_{\mathclap{\substack{i,j,k\in\{1,2,3\}\\\#\{i,j,k\}=3}}}\, \F_{p}(E_{i}, E_{j}, E_{k})=6\F_{p}(E_{1},E_{1},E_{2}).
		\end{align*}
		All in all we obtain
		\begin{align*}
			\MoveEqLeft
			\M_{p}^{1}(E)\\
			&=\Big(\sum_{\substack{i,j,k\in\{1,2,3\}\\\#\{i,j,k\}=1}}
			+\sum_{\substack{i,j,k\in\{1,3\}\\\#\{i,j,k\}=2}}
			+\sum_{\substack{i,j,k\in\{1,2\}\\\#\{i,j,k\}=2}}
			+\sum_{\substack{i,j,k\in\{2,3\}\\\#\{i,j,k\}=2}}
			+\sum_{\substack{i,j,k\in\{1,2,3\}\\\#\{i,j,k\}=3}}\Big)\F_{p}(E_{i}, E_{j}, E_{k})\\
			&=18\F_{p}(E_{1},E_{1},E_{2})=18\F_{p}(E_{2},E_{1},E_{1}).
		\end{align*}
	\textbf{Step 2} 
		Let us first choose parametrisations 
		\begin{align*}
			\gamma_{1}:[0,1]\to \R^{2},\,t\mapsto(-t,0)\quad\text{and}\quad\gamma_{2}:[0,1]\to \R^{2},\,t\mapsto(0,t)
		\end{align*}
		of $E_{1}$ and $E_{2}$, respectively. This gives us
		\begin{align*}
			\MoveEqLeft\F_{p}(E_{2},E_{1},E_{1})=
			\int_{0}^{1}\int_{0}^{1}\int_{0}^{1}\bigg(\frac{2x}{\sqrt{x^{2}+y^{2}}\sqrt{x^{2}+z^{2}}}\bigg)^{p}\dL^{1}(x)\dL^{1}(y)\dL^{1}(z)\\
			&\stackrel{\mathclap{\text{Lemma \ref{integral2}}}}{\leq}\quad
			\int_{0}^{1}\int_{0}^{1} 2^{p}\frac{\pi}{2^{p}}(zy)^{-(p-1)/2} \dL^{1}(y)\dL^{1}(z)\\
			&=\pi\int_{0}^{1} z^{(1-p)/2}\bigg[\frac{2}{3-p}y^{(3-p)/2}\bigg]_{0}^{1} \dL^{1}(z)
			=\pi\bigg[\frac{2}{3-p}z^{(3-p)/2}\bigg]_{0}^{1}\frac{2}{3-p}\\
			&=\frac{4\pi}{(3-p)^{2}}.
		\end{align*}
		Notice that the range $p\geq 2$ was neccessary to apply Lemma \ref{integral2}.
\end{proof}

\begin{corollary}[(The set $E$ has finite $\mathcal{M}_{p}^{1}$ for $p\in (0,3)$)]\label{finiteMpbutramificationpointsp<3}
	For $p\in (0,3)$ we have $\mathcal{M}_{p}^{1}(E)<\infty$.
\end{corollary}
\begin{proof}
	This is a consequence of Lemma \ref{finiteMpbutramificationpoints} 
	together with $\HM^{1}(E)=3$ and Lemma \ref{comparisonofcurvatureenergiesfordifferentp}.
\end{proof}

\section{Exponents are sharp and weak approximate tangents are optimal for \texorpdfstring{$\he=1$}{alpha=1}}\label{sectionresultsharp}

The exponents in the previous results are sharp, i.e.

\begin{lemma}[(A set with no appr. $1$-tangent and finite $\mathcal{U}_{(0,1)}^{1}$, $\mathcal{I}_{(0,2)}^{1}$ and $\mathcal{M}_{(0,3)}^{1}$)]\label{summaryforsetE}
	Let $E$ be the set from Definition \ref{definitionTset}. Then
	\begin{itemize}
		\item
			$E$ does not have a weak approximate $1$-tangent at $0$,
		\item
			$\U_{p}^{1}(E)<\infty$ for all $p\in (0,1)$,
		\item
			$\mathcal{I}_{p}^{1}(E)<\infty$ for all $p\in (0,2)$,
		\item
			$\M_{p}^{1}(E)<\infty$ for all $p\in (0,3)$.
	\end{itemize}
\end{lemma}
\begin{proof}
	This is Lemma \ref{TsetfiniteUp}, Corollary \ref{TsetfiniteIp<2} and Corollary \ref{finiteMpbutramificationpointsp<3}.
\end{proof}

The weak approximate $1$-tangents in the results for $\mathcal{I}_{p}^{1}$ and $\mathcal{M}_{p}^{1}$ are optimal in the following sense

\begin{lemma}[(A set with no appr. tangent and finite $\mathcal{I}_{p}^{1}$ for all $p\in (0,\infty)$)]\label{noapprtangentbutfiniteenergy}
	Set $a_{n}\vcentcolon=2^{-n^{n}n^{3}}$, $A_{n}\vcentcolon=[a_{n}/2,a_{n}]$ and
	\begin{align*}
		F\vcentcolon=\Big[\bigcup_{n\in\N}\underbrace{A_{2n}\times\{0\}}_{=\vcentcolon B_{2n}}\Big]
		\cup\Big[\bigcup_{n\in\N}\underbrace{\{0\}\times A_{2n-1}}_{=\vcentcolon B_{2n-1}}\Big].
	\end{align*}
	Then 
	\begin{itemize}
		\item
			$F$ does not have an approximate $1$-tangent at $0$,
		\item
			$1/\Delta[F]=\infty$,
		\item
			$\U_{p}^{1}(F)=\infty$ for all $p\in [1,\infty)$,
		\item
			$\mathcal{I}_{p}^{1}(F)<\infty$ for all $p\in (0,\infty)$,
		\item
			$\M_{p}^{1}(F)<\infty$ for all $p\in (0,\infty)$.
	\end{itemize}
\end{lemma}
\begin{proof}
	\textbf{Step 1}
		For $l\not=k$ we denote $\mu\vcentcolon=\min\{k,l\}$ and $M\vcentcolon=\max\{k,l\}$. Then
		\begin{align*}
			\MoveEqLeft\dist(B_{k},B_{l})\geq \dist(A_{k},A_{l})=2^{-(\mu^{\mu}\mu^{3}+1)}-2^{-M^{M}M^{3}}\\
			&= 2^{-(\mu^{\mu}\mu^{3}+1)}(1-2^{(\mu^{\mu}\mu^{3}+1)-M^{M}M^{3}})
			\geq 2^{-(\mu^{\mu}\mu^{3}+2)}=a_{\mu}/4.
		\end{align*}
		Let $y\in B_{k}$, $z\in B_{l}$ with $k\not= l$. Then
		\begin{align*}
			\kappa_{i}(y,z)\leq \frac{2}{\dist(B_{k},B_{l})}\leq \frac{8}{a_{\mu}}=\frac{8}{a_{\min\{k,l\}}}=\frac{8}{\max\{a_{k},a_{l}\}}.
		\end{align*}
	\textbf{Step 2}
		Let $q>1$.
		We now compute some inequalities for the indices. Let $k,m\in\N$, $k<m$, i.e. $m=k+i$ for some $i\in\N$. Then
		\begin{align*}
			m^{3}=(k+i)^{3}=k^{3}+3k^{2}i+3ki^{2}+i^{3},
		\end{align*}
		so that
		\begin{align}\label{indexinequality}
			-m^{3}+k^{3}=-(3k^{2}i+3ki^{2}+i^{3})\leq -3(k+i)=-3m.
		\end{align}
		As $qk^{k}\leq m^{m}$ for $1<q\leq k< m$ we have
		\begin{align*}
			-m^{m}m^{3}+qk^{k}k^{3}\leq -qk^{k}m^{3}+qk^{k}k^{3}=qk^{k}(-m^{3}+k^{3})
			\stackrel{(\text{\ref{indexinequality})}}{\leq} qk^{k}(-3m)\leq -3m.
		\end{align*}
		Consequently for all $1<q\leq k< m$
		\begin{align}\label{quotientofans}
			\frac{a_{m}}{a_{k}^{q}}=\frac{2^{-m^{m}m^{3}}}{2^{-qk^{k}k^{3}}}=2^{-m^{m}m^{3}+qk^{k}k^{3}}
			\leq 2^{-3m}.
		\end{align}
	\textbf{Step 3}
		As $\HM^{1}(B_{n})=a_{n}/2$ we have for $p\geq 3$, and $q=p-1>1$
		\begin{align*}
			\MoveEqLeft \sum_{\substack{k,m\in\N\\ k\not= m}}\int_{B_{k}}\int_{B_{m}}\kappa_{i}^{p}(y,z)\dd\HM^{1}(y)\dd\HM^{1}(z)\\
			&\leq \sum_{\substack{k,m\in\N\\ k\not= m}}\Big[ \frac{8}{\max\{a_{k},a_{m}\}}\Big]^{p}\frac{a_{k}a_{m}}{4}\\
			&\leq \frac{2\cdot 8^{p}}{4}\sum_{\substack{k,m\in\N\\1\leq k< m}} \frac{a_{k}a_{m}}{\max\{a_{k},a_{m}\}^{p}}\\
			&\leq 4\cdot 8^{p-1}\sum_{\substack{1\leq k\leq q\\k< m}} \frac{a_{k}a_{m}}{\max\{a_{k},a_{m}\}^{p}}
			+4\cdot 8^{p-1}\sum_{\substack{k,m\in\N\\q\leq k< m}} \frac{a_{m}}{a_{k}^{p-1}}\\
			&\leq 4\cdot 8^{p-1}\sum_{\substack{1\leq k\leq q\\k< m}} \frac{a_{k}a_{m}}{a_{\lceil q\rceil}^{p}}
			+4\cdot 8^{p-1}\sum_{\substack{k,m\in\N\\q\leq k< m}} \frac{a_{m}}{a_{k}^{q}}\\
			&\stackrel{\mathclap{\text{(\ref{quotientofans})}}}{\leq} \frac{4\cdot 8^{p-1}}{a_{\lceil q\rceil}^{p}}\sum_{k,m\in\N} 2^{-k}2^{-m}
			+4\cdot 8^{p-1}\sum_{\substack{k,m\in\N\\q\leq k< m}} 2^{-3m}\\
			&\leq \frac{4\cdot 8^{p-1}}{a_{\lceil q\rceil}^{p}}
			+4\cdot 8^{p-1}\sum_{\substack{k,m\in\N\\q\leq k< m}} 2^{-k}2^{-m}\\
			&\leq \frac{4\cdot 8^{p-1}}{a_{\lceil q\rceil}^{p}}
			+4\cdot 8^{p-1}\sum_{k,m\in\N} 2^{-k}2^{-m}\\
			&=4\cdot 8^{p-1}\Big(\frac{1}{a_{\lceil q\rceil}^{p}}+1\Big).
		\end{align*}
	\textbf{Step 4}
		Let $y,z\in B_{n}$. Then $\kappa(x,y,z)>0$ if and only if $x\in B_{k}$ for $(k-n)\,\mathrm{mod}\,2=1$. 
		To simplify matters we may without loss of generality assume that
		$k$ is even and $n$ is odd. We now have, compare Remark \ref{differentformulaskappa},
		\begin{align*}
			\kappa(x,y,z)=\frac{2\xi}{\sqrt{\xi^{2}+\eta^{2}}\sqrt{\xi^{2}+\zeta^{2}}},
		\end{align*}
		where we denote the non-zero entries of $x,y$ and $z$ by $\xi,\eta$ and $\zeta$, respectively.
		If we set $f(\xi)\vcentcolon= \kappa(x,y,z)/2$ for fixed $y$ and $z$ we have
		\begin{align*}
			\MoveEqLeft f'(\xi)=\frac{1}{\sqrt{\xi^{2}+\eta^{2}}\sqrt{\xi^{2}+\zeta^{2}}}
			-\frac{\xi^{2}}{\sqrt{\xi^{2}+\eta^{2}}^{3}\sqrt{\xi^{2}+\zeta^{2}}}
			-\frac{\xi^{2}}{\sqrt{\xi^{2}+\eta^{2}}\sqrt{\xi^{2}+\zeta^{2}}^{3}}\\
			&=\frac{(\xi^{2}+\eta^{2})(\xi^{2}+\zeta^{2})}{\sqrt{\xi^{2}+\eta^{2}}^{3}\sqrt{\xi^{2}+\zeta^{2}}^{3}}
			-\frac{\xi^{2}(\xi^{2}+\zeta^{2})+\xi^{2}
			(\xi^{2}+\eta^{2})}{\sqrt{\xi^{2}+\eta^{2}}^{3}\sqrt{\xi^{2}+\zeta^{2}}^{3}}\\
			&=\frac{(\xi^{2}+\eta^{2})\zeta^{2}-\xi^{2}(\xi^{2}+\zeta^{2})}
			{\sqrt{\xi^{2}+\eta^{2}}^{3}\sqrt{\xi^{2}+\zeta^{2}}^{3}}
			=\frac{\eta^{2}\zeta^{2}-\xi^{4}}{\sqrt{\xi^{2}+\eta^{2}}^{3}\sqrt{\xi^{2}+\zeta^{2}}^{3}},
		\end{align*}
		which is $0$ if and only if $\xi=\sqrt{\eta\zeta}$, because $\xi,\eta,\zeta>0$. 
		That $f$ attains its maximum at $\xi=\sqrt{\eta\zeta}$ 
		is clear by $f'\geq 0$ on $[0,\sqrt{\eta\zeta}]$ and
		$f'\leq 0$ on $[\sqrt{\eta\zeta},\infty)$. Since $\sqrt{\eta\zeta}\in A_{n}$ we have $(\sqrt{\eta\zeta},0)\not\in F$, as $n$ is odd,
		so that $\kappa_{i}(y,z)=\sup_{x\in F}\kappa(x,y,z)$ is attained for $x=(\xi,0)$, $\xi\in \{a_{n+1},a_{n-1}/2\}$. We have
		\begin{align*}
			f(a_{n+1})=\frac{a_{n+1}}{\sqrt{a_{n+1}^{2}+\eta^{2}}\sqrt{a_{n+1}^{2}+\zeta^{2}}}
			\leq \frac{a_{n+1}}{a_{n+1}^{2}+a_{n}^{2}/4}\leq 4\frac{a_{n+1}}{a_{n}^{2}}
		\end{align*}
		and
		\begin{align*}
			f(a_{n-1}/2)=\frac{a_{n-1}/2}{\sqrt{a_{n-1}^{2}/4+\eta^{2}}\sqrt{a_{n-1}^{2}/4+\zeta^{2}}}
			\leq \frac{a_{n-1}/2}{a_{n-1}^{2}/4+a_{n}^{2}/4}\leq 2\frac{a_{n-1}}{a_{n-1}^{2}}\leq \frac{4}{a_{n-1}}.
		\end{align*}
		As $2n^{n}n^{3}\leq (n+1)(n+1)^{n}(n+1)^{3}=(n+1)^{n+1}(n+1)^{3}$ and $a_{n-1}\leq 1$ we have $a_{n+1}a_{n-1}\leq a_{n}^{2}$ 
		and hence for $n\geq 2$
		\begin{align*}
			\kappa_{i}(y,z)=2\max\{f(a_{n+1}),f(a_{n-1}/2)\}
			\leq 2\max\Big\{\frac{4a_{n+1}}{a_{n}^{2}},\frac{4}{a_{n-1}}\Big\}=\frac{8}{a_{n-1}}.
		\end{align*}
		Consequently we have for $p\geq 3$
		\begin{align*}
			\MoveEqLeft \sum_{n=1}^{\infty}\int_{B_{n}}\int_{B_{n}}\kappa_{i}^{p}(y,z)\dd\HM^{1}(y)\dd\HM^{1}(z)\\
			&\leq \frac{2^{p}}{\dist(B_{1},\R\times\{0\})^{p}}\Big(\frac{1}{8}\Big)^{2}+
			\sum_{n=2}^{\infty}\frac{8^{p}}{a_{n-1}^{p}}\HM^{1}(B_{n})\HM^{1}(B_{n})\\
			&\leq\frac{2^{p}}{(1/4)^{p}}\Big(\frac{1}{8}\Big)^{2}+\sum_{n=2}^{\infty}\frac{8^{p}}{a_{n-1}^{p}}\frac{a_{n}^{2}}{4}
			\leq\frac{8^{p}}{64}+8^{p}\sum_{n=2}^{\infty}\frac{a_{n}}{a_{n-1}^{p}}\\
			&\leq \frac{8^{p}}{64}+8^{p}\sum_{n=2}^{\lceil p\rceil+1}\frac{a_{n}}{a_{n-1}^{p}}+
			8^{p}\sum_{n=\lceil p\rceil+1}^{\infty}\frac{a_{n}}{a_{n-1}^{p}}\\
			&\stackrel{\mathclap{\text{(\ref{quotientofans})}}}{\leq} C_{p}+8^{p}\sum_{n=\lceil p\rceil+1}^{\infty}2^{-3n}
			\leq C_{p}+8^{p}\sum_{n=0}^{\infty}2^{-n}\leq C_{p}+8^{p}\cdot 2.
		\end{align*}
	\textbf{Step 5}
		For $p\geq 3$ we now conclude that by Lemma \ref{decompositionoftripleintegral} we have
		\begin{align*}
			\MoveEqLeft \mathcal{I}_{p}^{1}(F)\leq \sum_{k,l\in\N}\int_{B_{k}}\int_{B_{l}}\kappa_{i}^{p}(y,z)\dd\HM^{1}(y)\dd\HM^{1}(z)\\
			&=\sum_{\substack{k,l\in\N\\ k\not= l}}\int_{B_{k}}\int_{B_{l}}\kappa_{i}^{p}(y,z)\dd\HM^{1}(y)\dd\HM^{1}(z)
			+\sum_{n\in\N}\int_{B_{n}}\int_{B_{n}}\kappa_{i}^{p}(y,z)\dd\HM^{1}(y)\dd\HM^{1}(z)<\infty
		\end{align*}
		Using $\HM^{1}(F)\leq 2$ together with Lemma \ref{comparisonofcurvatureenergiesfordifferentp} 
		we have $\mathcal{I}_{p}^{1}(F)<\infty$ for all $p\in (0,\infty)$.\\
	\textbf{Step 6}
		In Example \ref{weakbutnoapptangents} we have already seen that $F$ has no approximate tangent at $0$. 
		This observation combined with Lemma \ref{finiteUpguaranteesapproxiamtetangents} directly gives us 
		$1/\Delta[F]=\infty$ and $\U_{p}^{1}(F)=\infty$ for all $p\in [1,\infty)$. For $\M_{p}^{1}(F)<\infty$ for all $p\in (0,\infty)$ we consult Lemma
		\ref{inequalityforcurvatureenergies} together with $\HM^{1}(F)\leq 2$.
\end{proof}

\begin{appendix}

\section{Semi-continuous and measurable functions}\label{sectionsemicontinuousfunctions}

\begin{lemma}[(Metric is continuous)]\label{metriccontinuous}
	The mapping $f: X^{3}\to\R,\,(x,y,z)\mapsto d(x,y)$ is continuous.
\end{lemma}
\begin{proof}
	A mapping from a metric space to a metric space is continuous iff it is sequentially continuous. Let $(x,y,z)\in X^{3}$ and
	\begin{align*}
		(x_{n},y_{n},z_{n})\xrightarrow[n\to\infty]{(X^{3},d^{3})} (x,y,z).
	\end{align*}
	Then $x_{n}\to x$ and $y_{n}\to y$ in $X$ for $n\to\infty$, which gives us
	\begin{align*}
		\abs{d(x,y)-d(x_{n},y_{n})}&\leq \abs{d(x,y)-d(x_{n},y)}+\abs{d(x_{n},y)-d(x_{n},y_{n})}\\
		&\leq d(x,x_{n})+d(y,y_{n})\xrightarrow[n\to\infty]{} 0.
	\end{align*}
\end{proof}

\begin{lemma}[(Reciprocal of semi-continuous functions)]\label{reciprocalofsemi-continuousfunctions}
	Let $f:(X,d)\to \overline\R$, $f\geq 0$ be lower [upper] semi-continuous then $1/f$ is upper [lower] semi-continuous, 
	if we set $1/0=\infty$ and $1/\infty=0$.
\end{lemma}
\begin{proof}
	A function $f$ is lower semi-continuous if and only if for all $t\in\R$ the set $\{f\leq t\}$ is closed, see \cite[Remark 1.3, p.21]{Braides2002a}.
	Hence the sets $\{1/t\leq 1/f\}$ are closed for $t>0$,
	as is $\{\alpha\leq 1/f\}=X$ for $\alpha\leq 0$. The other case follows analogously.
\end{proof}

\begin{lemma}[(Semi-continuous functions are measurable)]\label{semicontiniuousfunctionsaremeasurable}
	Let $f:(X,d)\to\overline\R$ be upper or lower semi-continuous then $f$ is $\B(X)\mhyphen\B(\overline\R)$ measurable.
\end{lemma}
\begin{proof}
	If $f$ is lower semi-continuous then the set $\{f\leq t\}$ is closed for all $t\in\R$ and if $f$ is upper semi-continuous 
	then the set $\{t\leq f\}$ is closed and hence a Borel set.
\end{proof}

\begin{lemma}[(Positive powers of positive, s.c. functions are s.c.)]\label{powersofsemicontiniuous}
	Let $f:(X,d)\to\overline\R$, $f\geq 0$ be lower [upper] semi-continuous then for all $p\in (0,\infty)$ 
	the function $f^{p}$ is lower [upper] semicontinuous.
\end{lemma}
\begin{proof}
	Without loss of generality let $f\geq 0$ be lower semi-continuous. We have $\{f\leq t\}=\emptyset$ for $t<0$. If $t\geq 0$ we clearly have
	$f(x)^{p}\leq t\Leftrightarrow f(x)\leq t^{1/p}$.
\end{proof}

\begin{lemma}[(Measurability of piecewise functions)]\label{partiallydefinedmeasurable}
	Let $(X_{1},\A_{1})$, $(X_{2},\A_{2})$ be measuring spaces, $A\in\A_{1}$ and $f:A\to X_{2}$ be $\A_{1}\vert_{A}\mhyphen\A_{2}$ measurable and 
	$g:X_{1}\backslash A\to X_{2}$ be $\A_{1}\vert_{X_{1}\backslash A}\mhyphen\A_{2}$ measurable. Then
	\begin{align*}
		F:X_{1}\to X_{2},\,x\mapsto\begin{cases}f(x),&x\in A,\\g(x),&x\in X_{1}\backslash A,\end{cases}
	\end{align*}
	is $\A_{1}\mhyphen\A_{2}$ measurable.
\end{lemma}
\begin{proof}
	Let $E\in \A_{2}$, then there exist measurable sets $B,C\in\A_{1}$ such that
	\begin{align*}
		F^{-1}(E)=f^{-1}(E)\cup g^{-1}(E)=(B\cap A)\cup (C\cap [X_{1}\backslash A])\stackrel{A\in\A_{1}}{\in}\A_{1}.
	\end{align*}
\end{proof}

\begin{lemma}[(Extension of lower semi-continuous functions)]\label{extensionoflowersemicontinuousfunctions}
	Let $(X,d)$ be a metric space, $C\subset X$ closed and $f:X\backslash C\to\overline \R$, $f\geq 0$ lower semi-continuous. Then
	\begin{align*}
		\tilde f: X\to\overline\R,\,x\mapsto\begin{cases}f(x),&x\in X\backslash C,\\ 0,&x\in C, \end{cases}
	\end{align*}
	is lower semi continuous.
\end{lemma}
\begin{proof}
	Let $(x_{n})_{n\in\N}\subset X$ be a sequence converging to $x\in X$. If $x_{n}\in C$ for infinitely many $n\in\N$ 
	the sequence with these indices is contained
	in $C$ and converges to $x$, so that, since $C$ is closed, we have $x\in C$ and consequently
	\begin{align*}
		\tilde f(x)=0\leq\liminf_{n\to\infty}\tilde f(x_{n}).
	\end{align*}
	If $x_{n}\in C$ only for a finite number of $n\in\N$, we can use the lower semi-continuity of $f$ on $X\backslash C$ to get
	\begin{align*}
		\tilde f(x)=f(x)\leq \liminf_{n\to\infty}f(x_{n})= \liminf_{n\to\infty}\tilde f(x_{n}) \quad\text{if }x\in X\backslash C
	\end{align*}
	and
	\begin{align*}
		\tilde f(x)=0\leq \liminf_{n\to\infty}\underbrace{\tilde f(x_{n})}_{\geq 0} \quad\text{if }x\in C.
	\end{align*}
\end{proof}

\begin{lemma}[($\mathrm{diag}(X)$ and $X_{0}$ are closed)]\label{DeltaXandX0areclosed}
	Let $(X,d)$ be a metric space. Then the diagonal $\mathrm{diag}(X)$ and $X_{0}$ are closed sets.
\end{lemma}
\begin{proof}
	\textbf{Step 1}
		Let $((x_{n},y_{n}))_{n\in\N}\subset \mathrm{diag}(X)$ be a sequence converging to $(x,y)\in X^{2}$.
		Then $x_{n}=y_{n}$ and since convergence in the product space
		implies convergence of the projections we have $x_{n}=y_{n}\to x=y$, where we have used, that in Hausdorff spaces limits are unique.\\
	\textbf{Step 2}
		The set $X_{0}$ is closed in the product space, because let $((x_{n},y_{n},z_{n}))_{n\in\N}\subset X_{0}$ be a sequence converging to 
		$(x,y,z)\in X^{3}$. Since $X_{0}$
		is the union of the three sets
		\begin{align}\label{threeDeltatimesXlikesets}
			\mathrm{diag}(X)\times X,\quad\{(x,y,x)\in X^{3}\mid x,y\in X\}\quad\text{and}\quad X\times \mathrm{diag}(X)
		\end{align}
		there exists a subsequence converging to the same limit, which is contained in one of these sets. 
		Clearly these sets are closed, so that $X_{0}$ is closed.
\end{proof}

\begin{lemma}[(Cauchy sequence in complete metric spaces)]\label{lemmacauchysequencescompletespace}
	Let $(X,d)$ be a complete metric space, $(x_{n})_{n\in\N}\subset X$ with
	\begin{align*}
		d(x_{n},x_{n+1})\leq a_{n}\quad\text{and}\quad
		\sum_{n=1}^{\infty}a_{n}<\infty.
	\end{align*}
	Then there is $x\in X$, such that $x_{n}\to x$ and 
	\begin{align*}
		d(x_{n},x)\leq \sum_{i=n}^{\infty}a_{i}.
	\end{align*}
\end{lemma}
\begin{proof}
	Let $\epsilon>0$ and $N$ be large enough for $\sum_{i=N}^{\infty}a_{i}\leq \epsilon$, which is possible, as $\sum_{i=N}^{\infty}a_{i}\to 0$.
	Let $m,n\geq N$, without loss of generality $m>n$. Then
	\begin{align*}
		d(x_{n},x_{m})\leq \sum_{i=n}^{m-1}d(x_{i},x_{i+1})\leq \sum_{i=n}^{m-1}a_{i}\leq \sum_{i=n}^{\infty}a_{i}\leq \epsilon,
	\end{align*}
	so that $(x_{n})_{n\in\N}$ is a Cauchy sequence in the complete space $(X,d)$ and hence convergent. This means there is $x\in X$, such that $x_{n}\to x$. Then for all $N>n$ we have
	\begin{align*}
		d(x_{n},x)\leq \sum_{i=n}^{N-1}d(x_{i},x_{i+1})+d(x_{N},x)\leq \sum_{i=n}^{N-1}a_{i}+d(x_{N},x)
	\end{align*}
	and hence
	\begin{align*}
		d(x_{n},x)=\lim_{N\to\infty}d(x_{n},x)\leq \lim_{N\to\infty}\Big[ \sum_{i=n}^{N-1}a_{i}+d(x_{N},x) \Big]=\sum_{i=n}^{\infty}a_{i}.
	\end{align*}
\end{proof}

\section{Curvature energies under removal of acc. point}\label{sectionaccumulationpointremoved}

\begin{lemma}[($\kappa_{i}$ if accumulation point is removed)]\label{Raccpointremoved}
	Let $(X,d)$ be a metric space and $x\in X$. If $x$ is an accumulation point of $X$ then
	\begin{align*}
		\kappa_{i}^{X}(y,z)=\kappa_{i}^{X\backslash \{x\}}(y,z),\quad\text{for all }y,z\in X\backslash\{x\}, y\not=z.
	\end{align*}
\end{lemma}
\begin{proof}
	Let $y,z\in X\backslash\{x\}$, $y\not=z$. Then there is a sequence $(x_{n})_{n\in\N}\subset X\backslash\{y,z\}$, such that
	\begin{align*}
		\frac{1}{r(x_{n},y,z)}\xrightarrow[n\to\infty]{}\kappa_{i}(y,z).
	\end{align*}
	If there is a subsequence $x_{n_{k}}\not= x$ for all $k\in\N$ the proposition is clear, so we assume $x_{n}=x$ for all $n\geq N$. 
	We then find a sequence
	$\overline x_{n}\in X\backslash\{x,y,z\}$, such that $\overline x_{n}\to x$ and as $r$ is continuous, 
	see Lemma \ref{curvatureradiiareupersemi-continuous} (i), 
	this gives us
	\begin{align*}
		\frac{1}{r(\overline x_{n},y,z)}\xrightarrow[n\to\infty]{}\frac{1}{r(x,y,z)}=\kappa_{i}(y,z).
	\end{align*}
\end{proof}

\begin{lemma}[($\kappa_{G}$ if accumulation point is removed)]\label{RGaccpointremoved}
	Let $(X,d)$ be a metric space and $x\in X$. If $x$ is an accumulation point of $X$ then
	\begin{align*}
		\kappa^{X}_{G}(z)=\kappa^{X\backslash \{x\}}_{G}(z),\quad\text{for all }z\in X\backslash \{x\}.
	\end{align*}
\end{lemma}
\begin{proof}
	We may without loss of generality assume that $\# X\geq 3$, as otherwise $\kappa_{G}\equiv 0$ and $\kappa_{i}\equiv 0$ 
	for both $X$ and $X\backslash \{x\}$.
	Let $z\in X\backslash\{x\}$ then there are sequences $(x_{n})_{n\in\N}$ and $(y_{n})_{n\in\N}$ in $X\backslash \{z\}$ 
	with $x_{n}\not= y_{n}$ for all $n\in\N$, such that
	\begin{align*}
		\frac{1}{r(x_{n},y_{n},z)}\to \kappa_{G}(z).
	\end{align*}
	If there is a subsequence $(n_{k})_{k\in\N}$, such that $x_{n_{k}},y_{n_{k}}\not= x$ the proposition is clear. Let $x_{n_{l}}=x$ for all $l\in\N$. 
	For fixed $l$ there exists a sequence $(x_{k}^{l})_{k\in\N}$ with $x_{k}^{l}\to x$, $k\to\infty$, such that $x_{k}^{l}\not\in\{x,y_{n_{l}},z\}$.
	As $r$ is continuous, see Lemma \ref{curvatureradiiareupersemi-continuous} (i), and $\#\{x_{k}^{l},y_{n_{l}},z\}=3$ we have
	\begin{align}\label{limitinverseradiussubsequences}
		\frac{1}{r(x_{k}^{l},y_{n_{l}},z)}\xrightarrow[k\to\infty]{}\frac{1}{r(x,y_{n_{l}},z)}.
	\end{align}
	\textbf{Case 1}
		Assume $\kappa_{G}(z)<\infty$. Then for all $\epsilon>0$ there exists $l\in\N$ and $k\in\N$, such that
		\begin{align*}
			\MoveEqLeft \abs*{\kappa_{G}(z)-\frac{1}{r(x_{k}^{l},y_{n_{l}},z)}}
			\leq \abs*{\kappa_{G}(z)-\frac{1}{r(x,y_{n_{l}},z)}}+\abs*{\frac{1}{r(x,y_{n_{l}},z)}-\frac{1}{r(x_{k}^{l},y_{n_{l}},z)}}\\
			&\leq\epsilon+\epsilon=2\epsilon.
		\end{align*}
	\textbf{Case 2}
		Assume $\kappa_{G}(z)=\infty$. If there is an $l$, such that $1/r(x,y_{n_{l}},z)=\infty$, 
		then the proposition is clear by (\ref{limitinverseradiussubsequences}).
		We therefore assume that $1/r(x,y_{n_{l}},z)<\infty$ for all $l\in\N$. Then there exists a $K_{l}>0$ such that
		\begin{align*}
			\abs*{\frac{1}{r(x,y_{n_{l}},z)}-\frac{1}{r(x_{k}^{l},y_{n_{l}},z)}}\leq 1\quad\text{for all }k\geq K_{l}.
		\end{align*}
		Furthermore for all $M>0$ there is a $L_{M}>0$, such that
		\begin{align*}
			M\leq \frac{1}{r(x,y_{n_{l}},z)}\quad\text{for all }l\geq L_{M}.
		\end{align*}
		Hence for all $M>1$ there are $l_{0}$ and $k_{0}$, such that
		\begin{align*}
			\MoveEqLeft M-1\leq \abs*{\frac{1}{r(x,y_{n_{l}},z)}}-\abs*{\frac{1}{r(x,y_{n_{l}},z)}-\frac{1}{r(x_{k}^{l},y_{n_{l}},z)}}\\
			&\leq \frac{1}{r(x_{k}^{l},y_{n_{l}},z)},
		\end{align*}
		so that $\kappa_{G}^{X\backslash\{x\}}(z)=\infty$.
\end{proof}

\begin{lemma}[($\mathcal{F}$ if accumulation point is removed)]\label{curvatureenergiesifaccpointisremovedlemma}
	Let $(X,d)$ be a metric space, $x\in X$, $\he\in (0,\infty)$, $p\in (0,\infty)$ and $\mathcal{F}\in\{\U_{p}^{\he},\mathcal{I}_{p}^{\he},\M_{p}^{\he}\}$. 
	If $x$ is an accumulation point of $X$ then
	\begin{align*}
		\mathcal{F}(X)=\mathcal{F}(X\backslash\{x\}).
	\end{align*}
\end{lemma}
\begin{proof}
	For all metric spaces we have $\HM^{\he}(\{x\})=0$, because for all $\epsilon>0$ we can cover $\{x\}$ with $B_{\epsilon}(x)$, 
	which has diameter $2\epsilon$.
	From the definition of the integral it is now clear that for $\HM^{\he}$ 
	integration we can neglect singletons, hence
	together with Lemma \ref{Raccpointremoved} and Lemma \ref{RGaccpointremoved} we have $\mathcal{F}(X)=\mathcal{F}(X\backslash\{x\})$ 
	for all $p\in (0,\infty)$, if we recall  that by Lemma \ref{subspacehausdorffmeasure} we also have 
	$\HM^{\he}_{X\backslash\{x\}}=\HM^{\he}\rest_{X\backslash\{x\}}$.
\end{proof}

\begin{lemma}[(Hausdorff measure on subspaces)]\label{subspacehausdorffmeasure}
	Let $(X,d)$ be a metric space $A\subset X$ and $(A,d_{A})$ the associated metric subspace. Then for all $\he>0$
	\begin{align*}
		\HM^{\he}_{(A,d_{A})}=\HM^{\he}_{(X,d)} \rest A.
	\end{align*}
\end{lemma}
\begin{proof}
	Let $M\subset A$.\\
	\textbf{Step 1}
		Let $(C_{n})_{n\in\N}$ be a $\delta$-covering of $M$ in $(X,d)$, then $(D_{n})_{n\in\N}$ defined 
		by $D_{n}\vcentcolon=C_{n}\cap A$ is a $\delta$-covering of $M$ 
		in $(A,d_{A})$, such that
		\begin{align*}
			\diam_{A}(D_{n})\leq\diam_{X}(C_{n})\quad\text{for all }n\in\N,
		\end{align*}
		which implies
		\begin{align*}
			\HM^{\he}_{A,\delta}(M)\leq \HM^{\he}_{X,\delta}(M)
		\end{align*}
		and thereby ``$\leq$''.\\
	\textbf{Step 2}
		Let $(D_{n})_{n\in\N}$ be a $\delta$-covering of $M$ in $(A,d_{A})$, then $(D_{n})_{n\in\N}$ 
		is also a $\delta$-covering of $M$ in $(X,d)$, which gives us
		\begin{align*}
			\HM^{\he}_{X,\delta}(M)\leq \HM^{\he}_{A,\delta}(M)
		\end{align*}
		and therefore ``$\geq$''.
\end{proof}

\section{Estimate of integrals and $\dist(L_{x,y},0)$}\label{sectionestintegral}

\begin{lemma}[(Distance $\dist(L_{x,y},0)$ in terms of $\measuredangle(x,0,y)$)]\label{distanceintermsofangle}
	Let $x,y\in\R^{n}\backslash \{0\}$, $x\not= y$ such that $\epsilon\vcentcolon=\arccos(x\cdot y/(\norm{x}\norm{y}))\in(0,\pi)$ 
	and $L_{x,y}$ denote the straight line 
	connecting $x$ and $y$. Then
	\begin{align*}
		\dist(L_{x,y},0)\geq \frac{\sin(\epsilon)}{2} \min\{\norm{x},\norm{y}\}.
	\end{align*}
\end{lemma}
\begin{proof}
	Without loss of generality we might assume that $0,x,y\in\R^{2}$. Now we compute the area of the triangle given by $0,x,y$ as
	\begin{align*}
		\frac{1}{2}\sin(\epsilon)\norm{x}\norm{y}=\frac{1}{2}\norm{x-y}\dist(L_{x,y},0)
	\end{align*}
	and obtain
	\begin{align*}
		\dist(L_{x,y},0)=\sin(\epsilon)\frac{\norm{x}\norm{y}}{\norm{x-y}}
		\geq \sin(\epsilon)\frac{\norm{x}\norm{y}}{2\max\{\norm{x},\norm{y}\}}
		= \frac{\sin(\epsilon)}{2} \min\{\norm{x},\norm{y}\}.
	\end{align*}
\end{proof}

\begin{lemma}[(Integral I)]\label{integral2}
	For $y,z> 0$ and $p\geq 2$ we have
	\begin{align*}
		\int_{0}^{1}\frac{x^{p}}{(x^{2}+y^{2})^{p/2}(x^{2}+z^{2})^{p/2}}\dx\leq\frac{\pi}{2^{p}}(zy)^{-(p-1)/2}.
	\end{align*}
\end{lemma}
\begin{proof}
	We have
	\begin{align*}
		\MoveEqLeft \int_{0}^{1}\frac{x^{p}}{(x^{2}+y^{2})^{p/2}(x^{2}+z^{2})^{p/2}}\dx
		=\int_{0}^{1}\frac{x^{p}}{(x^{4}+(y^{2}+z^{2})x^{2}+y^{2}z^{2} )^{p/2}}\dx\\
		&\stackrel{\mathclap{y^{2}+z^{2}\geq 2yz}}{\leq}\quad \int_{0}^{1}\frac{x^{p}}{(x^{4}+2yzx^{2}+y^{2}z^{2} )^{p/2}}\dx
		=\int_{0}^{1}\frac{x^{p}}{(x^{2}+yz)^{2p/2}}\dx
		=\int_{0}^{1}\frac{x^{p}}{(x^{2}+yz)^{p}}\dx\\
		&=\int_{0}^{1}\frac{1}{(x+\frac{yz}{x})^{p}}\dx=\int_{0}^{1}\frac{1}{(x+\frac{yz}{x})^{2}}\frac{1}{(x+\frac{yz}{x})^{p-2}}\dx\\
		&\stackrel{\mathclap{x+zy/x\geq 2\sqrt{zy}}}{\leq}\quad\;\; \int_{0}^{1}\frac{1}{(x+\frac{yz}{x})^{2}}\frac{1}{(2\sqrt{zy})^{p-2}}\dx\\
		&\stackrel{\mathclap{\text{Lemma \ref{integral1}}}}{=}\quad\;\frac{1}{2^{p-2}}\frac{1}{(zy)^{p/2-1}}
		\frac{1}{2}\Bigg( \frac{\arctan{(\frac{1}{\sqrt{zy}}})}{\sqrt{zy}} -\underbrace{\frac{1}{1+zy}}_{\geq 0}\Bigg)
		\leq \frac{1}{2^{p-2}}\frac{1}{(zy)^{p/2-1}}
		\frac{1}{2} \frac{\pi}{2}\frac{1}{\sqrt{zy}}\\
		&=\frac{\pi}{2^{p}}(zy)^{-(p-1)/2}.
	\end{align*}
\end{proof}

\begin{lemma}[(Integral II)]\label{integral1}
	We have
	\begin{align*}
		\int_{0}^{1}\frac{1}{(x+\frac{zy}{x})^{2}}\dx=\frac{1}{2}\Bigg( \frac{\arctan{(\frac{1}{\sqrt{zy}}})}{\sqrt{zy}} -\frac{1}{1+zy}\Bigg).
	\end{align*}
\end{lemma}
\begin{proof}
	\begin{align*}
		\Bigg[\frac{1}{2}\Bigg( \frac{\arctan{(\frac{x}{\sqrt{zy}}})}{\sqrt{zy}} -\frac{x}{x^{2}+zy}\Bigg)\Bigg]'
		&=\frac{1}{2}\bigg( \underbrace{\frac{1}{\sqrt{zy}(1+(\frac{x}{\sqrt{zy}})^{2})}
		\frac{1}{\sqrt{zy}}}_{=\frac{1}{zy+x^{2}}} -\frac{1}{x^{2}+zy}+\frac{2x^{2}}{(x^{2}+zy)^{2}} \bigg)\\
		&=\frac{x^{2}}{(x^{2}+zy)^{2}}=\frac{1}{(x+\frac{zy}{x})^{2}}.
	\end{align*}
\end{proof}

\end{appendix}

\bibliography{smalllibrary.bib}{}
\bibliographystyle{amsalpha}
\bigskip
\noindent
\parbox[t]{.8\textwidth}{
Sebastian Scholtes\\
Institut f{\"u}r Mathematik\\
RWTH Aachen University\\
Templergraben 55\\
D--52062 Aachen, Germany\\
sebastian.scholtes@rwth-aachen.de}

\end{document}